\documentclass{amsart}
\usepackage{amssymb}

\newtheorem{thm}{Theorem}[section]
\newtheorem{prop}[thm]{Proposition}
\newtheorem{cor}[thm]{Corollary}

\newtheorem{rema}[thm]{Remark}

\numberwithin{equation}{section}

\newcommand{\nor}[0]{\Upsilon}
\newcommand{\sou}[0]{\Delta}

\begin{document}

\title[Automorphism groups of N=2 superconformal spheres]{Automorphism groups of N=2 superconformal super-Riemann spheres}

\author{Katrina Barron}
\address{Department of Mathematics, University of Notre Dame,
Notre Dame, IN 46556}
\email{kbarron@nd.edu}


\date{February 17, 2010}


\thanks{Supported by NSA grant MSPF-07G-169.}

\begin{abstract}
We determine the automorphism groups for the countably infinite family of N=2 superconformal equivalence classes of DeWitt N=2 superconformal super-Riemann surfaces with closed, genus-zero body.  We then analyze the Lie structure of these groups.   Under the correspondence between N=2 superconformal and N=1 superanalytic structures, the results extend to the determination of automorphism groups of N=1 superanalytic DeWitt super-Riemann surfaces with closed, genus-zero body.
\end{abstract}

\maketitle

\section{Introduction}

In \cite{B-uniformization} (cf. \cite{Manin1}, \cite{Manin2}), it is shown, in particular, that any N=2 superconformal super-Riemann surface with closed, genus-zero body is N=2 superconformally equivalent to one of a countably infinite family of uniformized N=2 superconformal super-Riemann spheres, denoted $S^2 \hat{\mathbb{C}}(n)$ for $n$ an integer.  In \cite{B-n2moduli} we analyzed the group of automorphisms for the equivalence class $S^2 \hat{\mathbb{C}}(0)$ as part of the larger study of the moduli space of certain N=2 superconformal super-Riemann spheres with tubes modeling worldsheet particle interactions in N=2 superconformal field theory.  In this paper, we determine the group of automorphisms for the remaining equivalence classes of genus-zero N=2 superconformal super-Riemann surfaces, and study their Lie structure.  In particular, we show that the super-dimension of the Lie supergroup of automorphisms of $S^2 \hat{\mathbb{C}}(n)$ has even dimension $4$ for $n \in \mathbb{Z}$, odd dimension $4$ for $|n|\leq 2$, and odd dimension $|n| + 2$ for $|n| \geq 2$.    

DeWitt N=2 superconformal super-Riemann surfaces are fiber bundles over Riemann surfaces with transition functions that satisfy certain properties called N=2 supersymmetry or N=2 superconformality.  These are the geometric structures underlying holomorphic, two-dimensional, N=2 superconformal field theory (cf. \cite{DPZ}, \cite{FMS},  \cite{LVW}, \cite{Wa}, \cite{Ge}).   As N=2 supersymmetric particles modeled as superstrings propagate through space-time, they sweep out an N=2 superconformal super-Riemann surface with half-infinite tubes called a worldsheet (see, for instance, \cite{B-n2moduli}).  Under certain meromorphicity conditions, the algebra of correlation functions governed by genus-zero nonsuper (resp. N=1 superconformal) worldsheet interactions has the structure of a vertex operator algebra (resp. an N=1 Neveu-Schwarz vertex operator superalgebra), \cite{H-book}, \cite{B-announce}-\cite{B-iso}.   However, the results of \cite{B-uniformization} indicate that the algebra of correlation functions governed by genus-zero N=2 superconformal worldsheets has as a substructure an N=2 Neveu-Schwarz vertex operator superalgebra \cite{B-axiomatic}, but in general will have substantially more structure.  This is due to the fact that, unlike in the nonsuper and N=1 super cases for which there is only one genus-zero surface up to global (N=1 super)conformal equivalence, in the N=2 super case, there is an infinite family of N=2 superconformally inequivalent genus-zero surfaces.   

To construct and study many of the aspects of N=2 superconformal field theory, one needs a description of the moduli space of genus-zero N=2 superconformal super-Riemann surfaces with half-infinite tubes attached; and to begin to study this structure, one needs to understand the genus-zero N=2 superconformal super-Riemann surfaces and their global N=2 superconformal automorphisms.  The determination and study of these automorphisms are the purpose of this paper.

There are two main approaches to supermanifolds: the ``concrete" or ``DeWitt" approach \cite{Fd}, \cite{D}, \cite{Ro}; and the ``ringed-space" approach \cite{Leites}, \cite{Manin1}, \cite{Manin2}.    The DeWitt approach and the ringed-space approach to supermanifolds are equivalent if one restricts the supermanifolds in the DeWitt approach to only allow for transition functions which do not include components that are odd functions of an even variable \cite{Batchelor}, \cite{Ro}.  However, for many applications to superconformal field theory, one needs to include these more general transition functions, which are naturally incorporated using the DeWitt approach.    Using the ringed-space approach, in order to incorporate the more general transition functions allowed in the DeWitt approach and in superconformal field theories, one must consider families of ringed-space supermanifolds over a given supermanifold.   For the purposes of this paper, we present our results in the concrete approach.   In particular, we are interested in extending the program developed by the author in \cite{B-announce}-\cite{B-iso} and \cite{B-change} in the N=1 superconformal case to the N=2 superconformal case.  This work in the N=1 case, developing a rigorous correspondence between the geometry of N=1 superconformal super-Riemann surfaces and the algebraic notion of N=1 Neveu-Schwarz vertex operator superalgebra, and applications to determining the change of variables formulas for N=1 Neveu-Schwarz vertex operator superalgebras, relied on the local N=1  infinitesimal superconformal transformations giving a representation of the N=1 Neveu-Schwarz algebra, which implied that local coordinates must have even and odd variables taking values in some underlying Grassmann algebra and involved component functions which are odd functions of an even variable.  The N=2 case will similarly depend on even and odd variables taking values in some underlying Grassmann algebra and odd superfunctions of an even variable.

Many consider the need for a choice of Grassmann algebra the main drawback of the concrete approach.  However within the concrete approach (in particular allowing for odd functions of an even variable), certain functorial properties can be incorporated into the definition of superfunction and supermanifold alleviating the need for a particular choice of underlying Grassmann algebra (cf. \cite{Schwarz1984}); or even further generalizations can be made by replacing the underlying Grassmann algebras with almost nilpotent superalgebras \cite{KS}.   Although we restrict most of our discussion in this paper to supermanifolds over Grassmann algebras, the extension to almost nilpotent algebras following \cite{KS} is a natural one, as is the functorial nature of our definitions; see Remarks \ref{functorial-remark} and \ref{functorial-remark2}.

We note here that in \cite{DRS}, it was shown that N=2 superconformal super-Riemann surfaces are equivalent to N=1 superanalytic super-Riemann surfaces.  The same correspondence was obtained earlier, as the authors of \cite{DRS} mention, in an unpublished letter of Deligne to Manin.  We discuss this equivalence in Section \ref{DRS-section} in the setting of concrete supermanifolds.  Under this correspondence the results of this paper extend to the equivalence classes of genus-zero N=1 superanalytic super-Riemann surfaces and their automorphism groups.  

\section{Preliminaries}\label{preliminaries-section}

In this section, we recall the notion of superalgebra, Lie superalgebra,  Grassmann algebra, superanalytic function, N=2 superconformal function, supermanifold and N=2 superconformal super-Riemann surface following, for instance \cite{B-n2moduli}, \cite{D}, \cite{Ro}.  

\subsection{Superalgebras}

Let $\mathbb{C}$ denote the complex numbers, let $\mathbb{Z}$ denote the integers, and let $\mathbb{Z}_2$ denote the integers modulo 2.  For a $\mathbb{Z}_2$-graded vector space $V = V^0 \oplus V^1$, over $\mathbb{C}$, define the {\it sign function} $\eta$ on the homogeneous subspaces of $V$ by $\eta(v) = j$, for $v \in V^j$ and $j \in \mathbb{Z}_2$.  If $\eta(v) = 0$, we say that $v$ is {\it even}, and if $\eta(v) = 1$, we say that $v$ is {\it odd}.    A {\it superalgebra} is an (associative) algebra $A$ (with identity $1 \in A$), such that: (i) $A$ is a $\mathbb{Z}_2$-graded algebra; (ii) $ab = (-1)^{\eta(a)\eta(b)} ba$ for $a,b$ homogeneous in $A$.

A $\mathbb{Z}_2$-graded vector space $\mathfrak{g} = \mathfrak{g}^0 \oplus \mathfrak{g}^1$ is said to be a {\it Lie superalgebra} if it has a bilinear operation $[\cdot,\cdot]$ on $\mathfrak{g}$ such that for $u,v$ homogeneous in $\mathfrak{g}$: (i) $\; [u,v] \in {\mathfrak g}^{(\eta(u) + \eta(v))\mathrm{mod} \; 2}$;  
(ii) skew symmetry holds $[u,v] = -(-1)^{\eta(u)\eta(v)}[v,u]$; (iii) the following Jacobi identity holds 
\[(-1)^{\eta(u)\eta(w)}[[u,v],w] + (-1)^{\eta(v)\eta(u)}[[v,w],u]+ \; (-1)^{\eta(w)\eta(v)}[[w,u],v] = 0. \]

For any $\mathbb{Z}_2$-graded associative algebra $A$ and for $u,v \in A$ of homogeneous sign, we can define $[u,v] = u v - (-1)^{\eta(u)\eta(v)} v u$, making $A$ into a Lie superalgebra.  The algebra of endomorphisms of $A$, denoted $\mbox{End} \; A$, has a natural $\mathbb{Z}_2$-grading induced {}from that of $A$, and defining $[X,Y] = X Y - (-1)^{\eta(X) \eta(Y)} Y X$ for $X,Y$ homogeneous in $\mbox{End} \; A$, this gives $\mbox{End} \; A$ a Lie superalgebra structure.  An element $D \in 
(\mbox{End} \; A)^j$, for $j \in \mathbb{Z}_2$, is called a {\it superderivation of sign $j$} (denoted $\eta(D) = j$) if $D$ satisfies the {\it super-Leibniz rule}
\begin{equation}\label{leibniz} 
D(uv) = (Du)v + (-1)^{\eta(D) \eta(u)} uDv  
\end{equation} 
for $u,v \in A$ homogeneous.

Let $\mathfrak{h}$ be a Lie superalgebra, and let $\mathrm{End} (\mathfrak{h})$ denote the Lie superalgebra of endomorphisms from $\mathfrak{h}$ to itself.  Then $\mathrm{End} (\mathfrak{h})$ is a Lie superalgebra.  We call an element $D \in \mathrm{End} (\mathfrak{h})$ a Lie superalgebra derivation if it satisfies 
\begin{equation}
D( [u,v] ) = [D(u), v] + (-1)^{\eta(D) \eta(u)} [u, D(v)]
\end{equation}
for all $u,v \in \mathfrak{h}$.   We denote the set of all Lie superalgebra derivations of $\mathfrak{h}$ by 
$\mathrm{Der} (\mathfrak{h})$, and note that $\mathrm{Der} (\mathfrak{h})$ is a Lie sub-superalgebra of $\mathrm{End} (\mathfrak{h})$.

Given two Lie superalgebras $\mathfrak{g}$ and $\mathfrak{h}$ and a Lie superalgebra homomorphism 
\begin{equation}
\sigma: \mathfrak{g} \longrightarrow \mathrm{Der} ( \mathfrak{h}),
\end{equation}
we can put a Lie superalgebra structure on $\mathfrak{g} \times \mathfrak{h}$ by defining
\begin{equation}
[u + v , u' + v']  = [u, u'] + \sigma_u(v') - (-1)^{\eta(v) \eta(u')} \sigma_{u'}(v) + [v, v'].
\end{equation} 
This is called the {\it semi-direct} product of $\mathfrak{g}$ with $\mathfrak{h}$ and is denoted $\mathfrak{g} \times_\sigma \mathfrak{h}$. 

Let $V$ be a vector space.  The exterior algebra generated by $V$, denoted  $\bigwedge (V)$, has the structure of a superalgebra.  Let $\mathbb{N}$ denote the nonnegative integers.  For $L \in \mathbb{N}$, fix $V_L$ to be  an $L$-dimensional vector space over $\mathbb{C}$ with basis $\{\zeta_1,\zeta_2, \ldots, \zeta_L\}$ such that $V_L \subset V_{L+1}$.  We denote $\bigwedge(V_L)$ by $\bigwedge_L$ and call this the {\it Grassmann algebra on $L$  generators}.  In other words, {}from now on we will consider the Grassmann algebras to have a fixed sequence of generators.  Note that $\bigwedge_L \subset \bigwedge_{L+1}$, and taking the direct limit as $L \rightarrow \infty$, we have the {\it infinite Grassmann algebra} denoted by $\bigwedge_\infty$.  Then $\bigwedge_L$ and $\bigwedge_\infty$ are the associative algebras over $\mathbb{C}$ with generators $\zeta_j$, for $j = 1, 2, \dots, L$ and $j = 1, 2, \dots$, respectively, and with relations $\zeta_j \zeta_k = - \zeta_k \zeta_j$,  for $j \neq k$, and $\zeta_j^2 = 0$.  We use the notation $\bigwedge_*$ to denote a Grassmann algebra, finite or infinite.

Let 
\begin{eqnarray*}
J^0_L \! \!  &=&  \! \! \bigl\{ (j) = (j_1, j_2, \ldots, j_{2n}) \; | \; j_1 < j_2 < \cdots < j_{2n}, \; j_l \in \{1, 2, \dots, L\}, \; n \in \mathbb{N} \bigr\}, \\ 
J^1_L  \! \! &=&  \! \! \bigl\{(j) = (j_1, j_2, \ldots, j_{2n + 1}) \; | \; j_1 < j_2 < \cdots < j_{2n + 1}, \; j_l \in \{1, 2, \dots, L\}, \;  n \in \mathbb{N} \bigr\},
\end{eqnarray*}
and $J_L = J^0_L \cup J^1_L$.  Let $\mathbb{Z}_+$ denote the positive integers, and let
\begin{eqnarray*}
J^0_\infty \! \! &=& \! \! \bigl\{(j) = (j_1, j_2, \ldots, j_{2n})\; | \; j_1 < j_2 < \cdots < j_{2n}, \; j_l \in \mathbb{Z}_+, \; n \in \mathbb{N} \bigr\}, \\
J^1_\infty \! \! &=& \! \! \bigl\{(j) = (j_1, j_2, \ldots, j_{2n + 1})\; | \; j_1 < j_2 < \cdots < j_{2n + 1}, \; j_l \in \mathbb{Z}_+, \; n \in \mathbb{N} \bigr\},
\end{eqnarray*}
and $J_\infty = J^0_\infty \cup J^1_\infty$.  We use $J^0_*$, $J^1_*$, and $J_*$ to denote $J^0_L$ or $J^0_\infty$, $J^1_L$ or $J^1_\infty$, and $J_L$ or $J_\infty$, respectively.  Note that $(j) = (j_1,\dots,j_{2n})$ for $n = 0$ is in $J^0_*$, and we denote this element by $(\emptyset)$.  The $\mathbb{Z}_2$-grading of $\bigwedge_*$ is given explicitly by
\begin{eqnarray*}
\mbox{$\bigwedge_*^0$} \! &=& \! \Bigl\{ \sum_{(j) \in J^0_*} a_{(j)}\zeta_{j_{1}}\zeta_{j_{2}} \cdots \zeta_{j_{2n}} \; \big\vert \; a_{(j)} \in \mathbb{C}, \; n \in \mathbb{N} \Bigr\}\\ 
\mbox{$\bigwedge_*^1$} \! &=& \! \Bigl\{ \sum_{(j) \in J^1_*} a_{(j)}\zeta_{j_{1}}\zeta_{j_{2}} \cdots \zeta_{j_{2n + 1}} \; \big\vert  \; a_{(j)} \in \mathbb{C}, \; n \in \mathbb{N}
\Bigr\} . 
\end{eqnarray*} 

We can also decompose $\bigwedge_*$ into {\it body}, $(\bigwedge_*)_B = \{ a_{(\emptyset)} \in \mathbb{C} \}$,  and {\it soul} 
\[(\mbox{$\bigwedge_*$})_S \; = \; \Bigl\{\sum_{(j) \in J_* \smallsetminus \{(\emptyset) \}} \! 
a_{(j)} \zeta_{j_1} \zeta_{j_2} \cdots \zeta_{j_n} \; \big\vert \; a_{(j)} \in \mathbb{C}  \Bigr\}\] 
subspaces such that $\bigwedge_* = (\bigwedge_*)_B \oplus (\bigwedge_*)_S$.  For $a \in \bigwedge_*$, we write $a = a_B + a_S$ for its body and soul decomposition.  We will use both notations $a_B$ and $a_{(\emptyset)}$ for the body of a supernumber $a \in \bigwedge_*$ interchangeably.  

For $n \in \mathbb{N}$, we introduce the notation $\bigwedge_{*>n}$ to denote a finite Grassmann algebra $\bigwedge_L$ with $L > n$ or an infinite Grassmann algebra.  We will use the corresponding index notations for the corresponding indexing sets $J^0_{*>n}, J^1_{*>n}$ and $J_{*>n}$.

\subsection{Superfunctions} 

For $n \in \mathbb{N}$, let  $U$ be a subset of $\bigwedge_*^0\oplus (\bigwedge_*^1)^n$.  A {\it $\bigwedge_*$-super-function} $H$ on $U$ in $(1,n)$-variables is a map $H: U \longrightarrow \bigwedge_*$,  $(z ,\theta_1,\dots, \theta_n) \mapsto H(z ,\theta_1,\dots, \theta_n)$
where $z$ is an even variable in $\bigwedge_*^0$ and $\theta_j$, for $j = 1,\dots, n$,  are odd variables in $\bigwedge_*^1$.  If $H$ takes values only in $\bigwedge_*^0$ (resp.  $\bigwedge_*^1$), we say that $H$ is an {\it even} (resp. {\it odd}) superfunction.  Let $f(z_B)$ be a complex analytic function in $z_B$.  For $z \in \bigwedge_*^0$, define
\begin{equation}\label{more-than-one-variable}
f(z) = \! \sum_{l \in \mathbb{N}} \frac{z_S^l }{l !} \biggl(\frac{\partial \; \;}{\partial z_B}\biggr)^l  f(z_B) .
\end{equation}

Consider the projection
\begin{equation} \label{projection-onto-body}
\pi^{(1,n)}_B : \mbox{$\bigwedge_*^0$} \oplus (\mbox{$\bigwedge_*^1$})^n  \longrightarrow  \mathbb{C}, \ \ \  \qquad (z, \theta_1,\dots, \theta_n) \mapsto  z_B. 
\end{equation}
Let $U \subseteq \bigwedge_{*>n-1}^0 \oplus (\bigwedge_{*>n-1}^1)^n$, and let $H$ be a $\bigwedge_{*>n-1}$-superfunction in $(1,n)$-variables defined on $U$. Then $H$ is said to be {\it superanalytic} if $H$ is of the form
\begin{equation}
H(z,\theta_1,\dots,\theta_n) = \sum_{ (j) \in J_n} \theta_{j_1} \cdots \theta_{j_{l}} f_{(j)}(z) , 
\end{equation} 
where each $f_{(j)}$ is of the form 
\begin{equation}\label{restrict-coefficients}
f_{(j)}(z) = \sum_{(k) \in J_{ * - n}} f_{(j),(k)}(z) \zeta_{k_1}\zeta_{k_2} \cdots \zeta_{k_{s}}, 
\end{equation} 
and each $f_{(j),(k)}(z_B)$ is analytic  in $z_B$ for $z_B \in  U_B = \pi_B^{(1,n)} (U) \subseteq \mathbb{C}$.   

We require the even and odd variables to be in $\bigwedge_{* > n-1}$, and we restrict the coefficients of the $f_{(j)}$'s to be in $\bigwedge_{* - n} \subseteq \bigwedge_{*>n-1}$ in order for the partial derivatives with respect to each of the odd variables to be well defined and for multiple partials to be well defined (cf. \cite{D}, \cite{B-memoir}, \cite{Ro}, \cite{B-n2moduli}).    In the language of \cite{Ro}, these $\bigwedge_{*>n-1}$-superfunctions are called $GC^\omega$ functions on $\mathbb{C}_S^{1,n}$ if $\bigwedge_* = \bigwedge_\infty$, and are called $GHC^\omega$ functions on 
$\mathbb{C}_{S[L]}^{1,n}$ if $\bigwedge_* = \bigwedge_L$ for $L\geq n$.   In particular, they subsume the class of $HC^\omega$ functions as defined in \cite{Ro}. 

We define the {\it DeWitt topology on $\bigwedge_{* }^0 \oplus (\bigwedge_{* }^1)^n$} by letting a subset $U$ of $\mbox{$\bigwedge_{*}^0$} \oplus (\mbox{$\bigwedge_{* }^1$})^n$ be an open set in the DeWitt topology if and only if $U = (\pi^{(1,n)}_B)^{-1} (V)$ for some open set $V \subseteq \mathbb{C}$.  Note that the natural domain of a superanalytic $\bigwedge_{* > n-1}$-superfunction in $(1,n)$-variables is an open set in the DeWitt topology.  

Let $(\bigwedge_*)^\times$ denote the set of invertible elements in $\bigwedge_*$.  Then $(\bigwedge_*)^\times = \{a \in \bigwedge_* \; | \; a_B \neq 0 \}$,  since 
$\frac{1}{a} = \frac{1}{a_B + a_S} = \sum_{n \in \mathbb{N}} \frac{(-1)^n a_S^n}{a_B^{n + 1}}$
is well defined if and only if $a_B \neq 0$. 

\begin{rema}\label{functorial-remark}
{\em
Recall that $\bigwedge_L \subset \bigwedge_{L+1}$ for $L \in \mathbb{N}$, and note that {}from (\ref{more-than-one-variable}), any superanalytic $\bigwedge_L$-superfunction, $H_L$, in $(1,n)$-variables for $L \geq n$ can naturally be extended to a superanalytic $\bigwedge_{L'}$-superfunction in $(1,n)$-variables for $L'>L$ and hence to a superanalytic $\bigwedge_\infty$-superfunction.  Conversely, if $H_{L'}$ is a superanalytic $\bigwedge_{L'}$-superfunction (or $\bigwedge_\infty$-superfunction) in $(1,n)$-variables for $L' >n$, then we can restrict $H_{L'}$ to a superanalytic $\bigwedge_L$-superfunction for $L'> L\geq n$ by restricting $(z,\theta_1,\dots, \theta_n) \in 
\bigwedge_L^0 \oplus (\bigwedge_L^1)^n$ and setting $f_{(j)} \equiv 0$ if $(j) \notin J_{L-n}$.  If $H_{L'}$ satisfies $f_{(j)} \equiv 0$ if $(j) \notin J_{L-n}$, then this restriction to $\bigwedge_L$ and then extension to $\bigwedge_{L'}$ results in the identity mapping, i.e., leaves $H_{L'}$ unchanged.  Thus any superanalytic function over $\bigwedge_{L'}$ in $(1,n)$-variables with coefficient functions $f_{(j)} = 0$ for $(j) \notin J_{L-n}$,  for $L\leq L'$, can be thought of as a functor from the category of Grassmann algebras $\bigwedge_*$ with $*\geq L+n$ to superanalytic functions over $\bigwedge_*$ in $(1,n)$-variables (cf. \cite{Schwarz1984}, \cite{KS}).   
 }
\end{rema}

\subsection{Superconformal $(1,2)$-superfunctions}\label{superconformal-section}

Let $z$ be an even variable in $\bigwedge_{*>1}^0$, and let $\theta^+$ and $\theta^-$ be odd variables in $\bigwedge_{*>1}^1$.   Define
\begin{equation}
D^\pm = \frac{\partial}{\partial \theta^\pm} + \theta^\mp \frac{\partial}{\partial z}.
\end{equation}
Then $D^\pm$ are odd superderivations on $\bigwedge_{*>1}$-superfunctions in $(1,2)$-variables which are superanalytic in some DeWitt open subset $U \subseteq \bigwedge_{*>1}^0 \oplus (\bigwedge_{*>1}^1)^2$.  Note that
\begin{eqnarray}
[D^\pm, D^\pm] & = & 2(D^\pm)^2 \ = \ 0\\ 
\left[D^+,D^-\right] &=& D^+D^- + D^-D^+ \ = \ 2 \frac{\partial}{\partial z} .
\end{eqnarray}

Let
\begin{eqnarray}
H : U \subseteq \mbox{$\bigwedge_{*>1}^0$} \oplus (\mbox{$\bigwedge_{*>1}^1$})^2
&\longrightarrow& \mbox{$\bigwedge_{*>1}^0$} \oplus (\mbox{$\bigwedge_{*>1}^1$})^2 \label{H-superanal}\\
(z,\theta^+,\theta^-) \ &\mapsto& (\tilde{z},\tilde{\theta}^+,\tilde{\theta}^-) \nonumber
\end{eqnarray} 
be superanalytic, i.e., $\tilde{z} = H^0(z, \theta^+, \theta^-)$ is an even superanalytic $(1,2)$-superfunction, and $\tilde{\theta}^\pm = H^\pm(z, \theta^+, \theta^-)$ are odd superanalytic $(1,2)$-superfunctions.   Then $D^+$ and $D^-$ transform under $H(z,\theta^+,\theta^-)$ by
\begin{equation}\label{transform-Dplus}
D^\pm = (D^\pm\tilde{\theta}^\pm)\tilde{D}^\pm + (D^\pm\tilde{\theta}^\mp) \frac{\partial}{\partial \tilde{\theta}^\mp} + (D^\pm\tilde{z} - \tilde{\theta}^\mp D^\pm\tilde{\theta}^\pm) \frac{\partial}{\partial \tilde{z}} .
\end{equation}

We define an {\it N=2  superconformal function} $H$ on a DeWitt open subset $U$ of $\bigwedge_{*>1}^0 \oplus (\bigwedge_{*>1}^1)^2$  to be a superanalytic $(1,2)$-superfunction from $U$ into $\bigwedge_{*>1}^0 \oplus (\bigwedge_{*>1}^1)^2$ under which $D^+$ and $D^-$ transform homogeneously of degree one.   That is, $H$ transforms $D^\pm$ by non-zero superanalytic functions times $\tilde{D}^\pm$, respectively.   Since such a superanalytic function $H(z,\theta^+,\theta^-) = (\tilde{z}, \tilde{\theta}^+,\tilde{\theta}^-)$  transforms $D^+$ and $D^-$ according to  (\ref{transform-Dplus}), $H$ is superconformal if and only if, in addition to being superanalytic, $H$ satisfies
\begin{eqnarray}
D^\pm \tilde{\theta}^\mp &=& 0, \label{basic-superconformal-condition1} \\
D^\pm\tilde{z} - \tilde{\theta}^\mp D^\pm\tilde{\theta}^\pm &=& 0, \label{basic-superconformal-condition4}
\end{eqnarray}
for $D^\pm \tilde{\theta}^\pm$ not identically zero, thus transforming $D^\pm$ by $D^\pm = (D^\pm \tilde{\theta}^\pm)\tilde{D}^\pm$.   These conditions imply that we can write $H(z,\theta^+,\theta^-) = (\tilde{z}, \tilde{\theta}^+, \tilde{\theta}^-)$ as
\begin{eqnarray}
\qquad \qquad \tilde{z} &=& f(z) + \theta^+ g^+(z) \psi^-(z) + \theta^- g^-(z) \psi^+(z) + \theta^+ \theta^- (\psi^+(z)\psi^-(z))' \label{superconformal-condition1} \\
\tilde{\theta}^\pm &=& \psi^\pm(z) + \theta^\pm g^\pm(z) \pm \theta^+ \theta^-(\psi^\pm)'(z) \label{superconformal-condition3}
\end{eqnarray}
for $f$, $g^\pm$ even and $\psi^\pm$ odd superanalytic $(1,0)$-superfunctions in $z$, satisfying the condition 
\begin{equation}\label{superconformal-condition4}
f'(z) \; = \; (\psi^+)'(z)\psi^-(z) - \psi^+(z) (\psi^-)'(z) + g^+(z) g^-(z) ,
\end{equation}
and we also require that $D^+ \tilde{\theta}^+$ and $D^- \tilde{\theta}^-$ not be identically zero. 

Thus  an N=2 superconformal function $H$ is uniquely determined by the superanalytic functions $f(z)$, $\psi^\pm(z)$,  and $g^\pm(z)$ satisfying the condition (\ref{superconformal-condition4}). 

\begin{rema} {\em The setting above is in the  ``homogeneous" coordinate system, denoted by N=2 supercoordinates $(z, \theta^+, \theta^-)$.    There is  another commonly used coordinate system for N=2 super-settings, namely the ``nonhomogeneous" coordinates, denoted by 
$(z, \theta_1, \theta_2)$, where $\theta_1 = \frac{1}{\sqrt{2}} \left( \theta^+ + \theta^- \right)$ and  $\theta_2 = - \frac{i}{\sqrt{2}} \left( \theta^+ - \theta^- \right)$, or equivalently $\theta^\pm  = \frac{1}{\sqrt{2}} \left( \theta_1 \pm  i\theta_2 \right)$.  This is a standard transformation in N=2 superconformal field theory (cf. \cite{DRS}, \cite{B-n2moduli}, \cite{B-axiomatic}, \cite{B-uniformization}), and the terms homogeneous and nonhomogeneous were introduced in \cite{B-n2moduli} along with reasons for this nomenclature.   These reasons lie in the fact that the $U(1)$-current algebra in the Lie superalgebra of infinitesimal N=2 superconformal transformations (see (\ref{Virasoro-relation})-(\ref{Neveu-Schwarz-relation-last}) below) acts homogeneously (resp. nonhomogeneously) on odd elements when in the homogeneous (resp. nonhomogeneous) coordinates.  We will continue to use the homogeneous coordinate system for the purposes of this paper, as the results are more easily presented in this system (cf. \cite{B-uniformization}). }
\end{rema}

\subsection{Complex supermanifolds and N=2 superconformal super-Riemann surfaces}  

A {\em DeWitt $(1,n)$-dimensional supermanifold over $\bigwedge_*$} is a topological space $X$ with a countable basis which is locally homeomorphic to an open subset of $\bigwedge_*^0 \oplus (\bigwedge_*^1)^n$ in the DeWitt topology.  A {\em DeWitt $(1,n)$-chart on $X$ over $\bigwedge_*$} is 
a pair $(U, \Omega)$ such that $U$ is an open subset of $X$ and $\Omega$ is a homeomorphism of $U$ onto an open subset of $\bigwedge_*^0 \oplus (\bigwedge_*^1)^n$ in the DeWitt topology.  A {\em superanalytic atlas of DeWitt $(1,n)$-charts on $X$ over $\bigwedge_{* > n-1}$} is a family of charts $\{(U_{\alpha}, \Omega_{\alpha})\}_{\alpha \in A}$ satisfying 
 
(i) Each $U_{\alpha}$ is open in $X$, and $\bigcup_{\alpha \in A} U_{\alpha} = X$. 

(ii) Each $\Omega_{\alpha}$ is a homeomorphism {}from $U_{\alpha}$ to an open set in $\bigwedge_{* > n-1}^0 \oplus (\bigwedge_{* > n-1}^1)^n$ in the DeWitt topology, such that $\Omega_{\alpha} \circ \Omega_{\beta}^{-1}: \Omega_{\beta}(U_\alpha \cap U_\beta) \longrightarrow \Omega_{\alpha}(U_\alpha \cap U_\beta)$ is superanalytic for all non-empty $U_{\alpha} \cap U_{\beta}$, i.e., $\Omega_{\alpha} \circ \Omega_{\beta}^{-1} = (\tilde{z}, \tilde{\theta}_1, \dots, \tilde{\theta}_n)$ where $\tilde{z}$ and $\tilde{\theta}_j$, for $j = 1,\dots, n$, are even and odd, respectively, superanalytic $\bigwedge_{* > n-1}$-superfunctions in $(1,n)$-variables.

Such an atlas is called {\em maximal} if, given any chart $(U, \Omega)$ such that $\Omega \circ \Omega_{\beta}^{-1} : \Omega_{\beta} (U \cap U_\beta) \longrightarrow \Omega (U \cap U_\beta)$ 
is a superanalytic homeomorphism for all $\beta$, then $(U, \Omega) \in \{(U_{\alpha}, \Omega_{\alpha})\}_{\alpha \in A}$.
 
A {\em DeWitt $(1,n)$-superanalytic supermanifold over $\bigwedge_{* > n-1}$} is a DeWitt $(1,n)$-dimensional supermanifold $M$ together with a maximal superanalytic atlas of DeWitt $(1,n)$-charts over $\bigwedge_{* > n-1}$.  In the language of \cite{Ro}, if $\bigwedge_*$ is infinite dimensional, then these supermanifolds are called $(1,n)$-$GC^\omega$ DeWitt supermanifolds.  If the transition functions for $M$ are restricted to be $HC^\omega$ functions, then we call $M$ an {\it $\mathcal{H}$-supermanifold}.  These are the supermanifolds that are most often studied in the ringed-space approach, cf.  \cite{Leites}, \cite{Manin1}, \cite{Manin2}, \cite{Ro}.

Given a DeWitt $(1,n)$-superanalytic supermanifold $M$ over $\bigwedge_{* > n-1}$, define an equivalence relation $\sim$ on M by letting $p \sim q$ if and only if there exists $\alpha \in A$ such that $p,q \in U_\alpha$ and $\pi_B^{(1,n)} (\Omega_\alpha (p)) = \pi_B^{(1,n)} (\Omega_\alpha (q))$
where $\pi_B^{(1,n)}$ is the projection given by (\ref{projection-onto-body}).  Let $p_B$ denote the equivalence class of $p$ under this equivalence relation.  Define the {\it body} $M_B$ of $M$ to be the 
complex manifold with analytic structure given by the coordinate charts $\{((U_\alpha)_B, (\Omega_\alpha)_B) \}_{\alpha \in A}$ where $(U_\alpha)_B = \{ p_B \; | \; p \in U_\alpha \}$, and $(\Omega_\alpha)_B : (U_\alpha)_B \longrightarrow \mathbb{C}$ is given by $(\Omega_\alpha)_B (p_B) = \pi_B^{(1,n)} \circ \Omega_\alpha (p)$. We define the genus of $M$ to be the genus of $M_B$.

Note that $M$ is a complex fiber bundle over the complex manifold $M_{B}$; the fiber is the complex vector space $(\bigwedge_{* > n-1}^0)_S \oplus (\bigwedge_{* > n-1}^1)^n$.  This bundle is not in general a vector bundle since the transition functions are in general nonlinear.  

An {\em N=2 superconformal super-Riemann surface over $\bigwedge_{*>1}$} is a DeWitt $(1,2)$-superanalytic supermanifold over $\bigwedge_{*>1}$ with coordinate atlas $\{(U_{\alpha}, \Omega_{\alpha})\}_{\alpha \in A}$ such that the coordinate transition functions $\Omega_{\alpha} \circ \Omega_{\beta}^{-1}$ in addition to being superanalytic are also N=2 superconformal for all non-empty $U_{\alpha} \cap U_{\beta}$.

Since the condition that the coordinate transition functions be N=2 superconformal instead of merely superanalytic is such a strong condition (unlike in the nonsuper case), we again stress the distinction between an N=2 superanalytic super-Riemann surface which has {\it superanalytic} transition functions versus an N=2 superconformal super-Riemann surface which has N=2 {\it superconformal} transition functions.  In the literature one will find the term ``super-Riemann surface" or ``Riemannian supermanifold" used for both merely superanalytic structures (cf. \cite{D}) and for superconformal structures (cf. \cite{Fd}).   

From now on, we will focus on N=1 superanalytic super-Riemann surfaces, that is DeWitt $(1,1)$-superanalytic supermanifolds over $\bigwedge_{*>0}$, and N=2 superconformal super-Riemann surfaces. 

Let $M_1$ and $M_2$ be N=2 superconformal (resp. N=1 superanalytic) super-Riemann surfaces with coordinate atlases $\{(U_{\alpha},$ $\Omega_{\alpha})\}_{\alpha \in A}$ for $M_1$ and $\{(V_{\beta}, \Xi_{\beta})\}_{\beta \in B}$ for $M_2$.  A map $F: M_1 \longrightarrow M_2$ is said to be {\it N=2 superconformal} (resp. {\it N=1 superanalytic})  if $\Xi_\beta \circ F \circ \Omega_\alpha^{-1}: \Omega_\alpha (U_\alpha \cap F^{-1} (V_\beta)) \longrightarrow \Xi_\beta(V_\beta)$ is N=2 superconformal (resp. N=1 superanalytic)  for all $\alpha \in A$ and $\beta \in B$ with $U_\alpha \cap F^{-1} (V_\beta) \neq \emptyset$.  If in addition, $F$ is bijective, then we say that $M_1$ and $M_2$ are {\it N=2 superconformally equivalent} (resp. {\it N=1 superanalytically equivalent}).  

\begin{rema}\label{functorial-remark2}
{\em  Just as a single superanalytic function over a certain Grassmann algebra can be thought of as a functor from a (sub)category of Grassmann algebras to superanalytic functions over any one of these Grassmann algebras (see Remark \ref{functorial-remark}), so can a DeWitt $(1,n)$-superanalytic supermanifold over a certain Grassmann algebra be thought of as a functor from a (sub)category of Grassmann algebras to DeWitt $(1,n)$-superanalytic supermanifolds over any one of these Grassmann algebras.  Let $M$ be an $(1,n)$-superanalytic supermanifold over $\bigwedge_{L'}$ for $L'\geq n$ with coordinate atlas given by $\{(U_\alpha, \Omega_\alpha)\}_{\alpha \in A}$.  If the coordinate transition functions for $M$ are such that the coefficient functions $f_{(j)} \equiv 0$ for $(j) \notin J_{L-n}$ for some $L'>L \geq n$, then the submanifold of $M$ given by $\bigcup_{\alpha \in A} \Omega_\alpha^{-1} ((\Omega_\alpha (U_\alpha))_B \times (\bigwedge_L)_S )$ is naturally a $(1,n)$-superanalytic supermanifold over $\bigwedge_L$.  Moreover, if $M_1$ and $M_2$ result in the same submanifold under this restriction from $\bigwedge_{L'}$ to $\bigwedge_L$, then $M_1 = M_2$.  Thus there is a natural and unique extension of any $(1,n)$-superanalytic supermanifold over $\bigwedge_L$ to a $(1,n)$-superanalytic supermanifold over $\bigwedge_*$ for $*>L$. 
}
\end{rema}

\subsection{The equivalence of N=2 superconformal and N=1 superanalytic DeWitt super-Riemann surfaces}\label{DRS-section}

In this section, we recall (and slightly extend) some results from \cite{DRS} establishing an equivalence between N=1 superanalytic  super-Riemann surfaces and N=2 superconformal super-Riemann surfaces.  Our main results in this paper, Theorems \ref{auto-groups-thm}, \ref{infinitesimals-thm}, and \ref{last-cor}  and Corollary \ref{infinitesimals-cor}, are formulated and proved for N=2 superconformal super-Riemann surfaces.  However, using Proposition \ref{DRS-prop} from this section, our results are easily formulated in the corresponding N=1 superanalytic setting.  

Although we follow \cite{DRS}, modified slightly to our setting, there are discrepancies between some of our formulas and those given in \cite{DRS}.  For instance, there is a typo in \cite{DRS} in the transformation from the nonhomogeneous coordinate system $(z, \theta_1, \theta_2)$ to the homogeneous coordinate system $(z, \theta^+, \theta^-)$;  the typo is a factor of $1/2$ erroneously introduced into the $D^\pm$ superderivations after the transformation of coordinates, and this factor is carried throughout their calculations. 

Let $U_B$ be an open set in $\mathbb{C}$.  Let  $\mathcal{SC}_{*>1}(2, U_B)$ be the set of invertible N=2 superconformal functions defined on the DeWitt open set $U_B \times ((\bigwedge_{*>1}^0)_S \oplus (\bigwedge_{*>1}^1)^2)$ in $\bigwedge_{*>1}^0 \oplus (\bigwedge_{*>1}^1)^2$.  Let $\mathcal{SA}_{*>1}(1, U_B)$ be the set of invertible N=1 superanalytic functions $H$ defined on the DeWitt open set $U_B \times (\bigwedge_{*>1})_S$ in $\bigwedge_{*>1}$ such that  the coefficients of the functions defining $H$ are restricted to lie in $\bigwedge_{*-2}$ rather than just in $\bigwedge_{*-1}$; that is in (\ref{restrict-coefficients}), we take $(k) \in J_{*-2}$ rather than $(k) \in J_{*-1}$.  

Define the map
\begin{eqnarray}
\mathcal{F}_1 : \mathcal{SC}_{*>1}(2, U_B) & \longrightarrow & \mathcal{SA}_{*>1}(1, U_B)\\
H & \mapsto & \mathcal{F}_1(H) \nonumber
\end{eqnarray}
as follows: For $H \in \mathcal{SC}_{*>1}(2, U_B)$, then $H(z, \theta^+, \theta^-) = (\tilde{z}, \tilde{\theta}^+, \tilde{\theta}^-)$ is of the form (\ref{superconformal-condition1})-(\ref{superconformal-condition4}) for even functions $f$ and $g^\pm$ and odd functions $\psi^\pm$.    Define 
\begin{equation}\label{N=2-to-N=1}
\mathcal{F}_1(H) (z, \theta) = (f(z) + \psi^+(z) \psi^-(z) +  2\theta g^+(z) \psi^-(z), \   \psi^+(z) + \theta g^+(z))
\end{equation}

This invertible N=1 superanalytic function $\mathcal{F}_1(H)$ can be thought of as arising from performing the N=2 superanalytic coordinate transformation 
\begin{equation}
(z, \ \theta^+, \ \theta^-) \mapsto (u, \ \eta,  \ \alpha) = (z + \theta^+ \theta^-, \ \theta^+, \ \theta^-).
\end{equation}
Under this transformation, we obtain the N=2 superanalytic function in the even variable $u$ and the two odd variables $\eta$ and $\alpha$ given by 
\begin{eqnarray}
\qquad \qquad \tilde{u} &=& f(u) + \psi^+(u) \psi^-(u) +  2\eta g^+(u) \psi^-(u) \\
\tilde{\eta} &=& \psi^+(u) + \eta g^+(u) \\
\tilde{\alpha} &=& \psi^-(u) + \alpha g^-(u) -2 \eta \alpha (\psi^-)'(u) .
\end{eqnarray}

Conversely, define the map
\begin{eqnarray}
\mathcal{F}_2 : \mathcal{SA}_{*>1}(1, U_B) & \longrightarrow & \mathcal{SC}_{*>1}(2, U_B)\\
H & \mapsto & \mathcal{F}_2(H)  \nonumber
\end{eqnarray}
as follows: For $H \in \mathcal{SA}_{*>1}(1, U_B)$, then $H(z, \theta) = (f_1(z) + \theta \xi(z), \ \psi(z) + \theta g(z))$ for even functions $f_1(z)$ and $g(z)$ and odd functions $\xi(z)$ and $\psi(z)$, and with $g(z)$ nonvanishing.  Define $\mathcal{F}_2(H) (z, \theta^+, \theta^-) =  (\tilde{z}, \tilde{\theta}^+, \tilde{\theta}^-)$ to be of the form (\ref{superconformal-condition1})-(\ref{superconformal-condition3}) where
\begin{eqnarray}
f(z) &=& f_1(z) - \frac{\psi(z)\xi(z)}{2 g(z)}, \\
g^+(z) &=& g(z), \qquad \ \ \mbox{and} \qquad \ g^-(z)\ \  = \ \ \frac{f_1'(z)}{g(z)} - \frac{\psi'(z) \xi(z)}{g(z)^2}\\
\psi^+(z) &=& \psi(z), \qquad\ \  \mbox{and} \qquad \psi^-(z) \ \ = \ \ \frac{\xi(z)}{2g(z)}.
\end{eqnarray}
One can easily check that condition (\ref{superconformal-condition4}) is satisfied, and thus $\mathcal{F}_2(H)$ is indeed N=2 superconformal. 

We have that $\mathcal{F}_1$ and $\mathcal{F}_2$ are bijections and 
\begin{equation}\label{inverse-functors}
\mathcal{F}_1 \circ \mathcal{F}_2 = id_{\mathcal{SA}_{*>1}(1, U_B)} \qquad  \mbox{and} \qquad \mathcal{F}_2 \circ \mathcal{F}_1 = id_{\mathcal{SC}_{*>1}(2, U_B)}.
\end{equation}

Let $\mathcal{SCM}_{*>1}(2)$ be the category of N=2 superconformal super-Riemann surfaces over the Grassmann algebra $\bigwedge_{*>1}$, and let $\mathcal{SAM}_{*>1}(1)$ be the category of N=1 superanalytic super-Riemann surfaces $M$ over the Grassmann algebra $\bigwedge_{*>1}$ such that the transition functions for $M$ are in $\mathcal{SA}_{*>1}(1, U_B)$ for some $U_B \in \mathbb{C}$. 

Define the functor 
\begin{eqnarray}\label{category-iso}
\mathcal{F}: \mathcal{SCM}_{*>1}(2) &\longrightarrow& \mathcal{SAM}_{*>1}(1)\\
M & \mapsto & \mathcal{F}(M) \nonumber
\end{eqnarray}
as follows:  Let $M$ be an N=2 superconformal super-Riemann surface over the Grassmann algebra $\bigwedge_{*>1}$ with coordinate atlas $\{ (U_\alpha, \Omega_\alpha)\}_{\alpha \in A}$.   Let $\mathcal{F}(M)$ be the N=1 superanalytic  super-Riemann surface with body $M_B$ obtained by patching together  DeWitt open domains in $\bigwedge_{*>1}$ with local coordinates $(z, \theta)$ by means of the transition functions $\mathcal{F}_1 ( \Omega_\alpha \circ \Omega_\beta^{-1}):  (\Omega_\beta(U_\alpha \cap U_\beta))_B \times (\bigwedge_{*>1})_S \longrightarrow (\Omega_\alpha (U_\alpha \cap U_\beta))_B \times (\bigwedge_{*>1})_S$.

From (\ref{inverse-functors}), it follows that $\mathcal{F}$ is an isomorphism of categories.  Thus we have the following proposition (cf. \cite{DRS}):
\begin{prop}\label{DRS-prop}
The category $\mathcal{SCM}_{*>1}(2)$  of N=2 superconformal super-Riemann surfaces over the Grassmann algebra $\bigwedge_{*>1}$ is isomorphic to the category $\mathcal{SAM}_{*>1}(1)$ of N=1 superanalytic super-Riemann surfaces such that  the coefficients of the coordinate transition functions are restricted to lie in $\bigwedge_{*-2}$.
\end{prop}

\begin{rema}
{\em  In N=2 superconformal field theory, the supermanifolds that arise from superstrings propagating through space-time, are N=2 superconformal super-Riemann surfaces with half-infinite tubes attached.  These half-infinite tubes are N=2 superconformally equivalent to punctures on the N=2 superconformal super-Riemann surface with N=2 superconformal local coordinates vanishing at the punctures.  Although, there is a bijection between N=2 superconformal local coordinates in a neighborhood of a given point on an N=2 superconformal super-Riemann surface $M$ and the N=1 superanalytic local coordinates in a neighborhood of the corresponding point on $\mathcal{F}(M)$, 
a bijective correspondence does not exist between such N=2 superconformal local coordinates vanishing at the point $p \in M$ and N=1 superanalytic local coordinates vanishing at the corresponding point in $\mathcal{F}(M)$.  For example, the N=1 superanalytic function $H(z, \theta) = ( z + \theta, \theta)$ vanishes at the origin $(0,0)$ of the N=1 superplane $\bigwedge_{*>1}$.  However the corresponding N=2 superconformal function $\mathcal{F}_1(H) (z, \theta^+, \theta^-) = (z + \frac{1}{2} \theta^+, \theta^+, \frac{1}{2} + \theta^-)$ does not vanish at the corresponding point in $\mathcal{F}^{-1} ( \bigwedge_{*>1}) = \bigwedge_{*>1}^0 \oplus (\bigwedge_{*>1}^1)^2$, that point being the origin.   Thus one cannot simply replace N=2 superconformal worldsheets swept out by propagating superstrings by N=1 superanalytic worldsheets when the full data of the propagating strings is included.  One must either work in the N=2 superconformal setting, or take into account the discrepancies that arise by using the N=1 superanalytic setting when modeling the incoming and outgoing tubes for the superstrings.  See, for example, \cite{B-axiomatic-deformations} for further discussion of this fact. }
\end{rema}

\section{The Uniformization Theorem for genus-zero N=2 superconformal and N=1 superanalytic super-Riemann surfaces}

For $n \in \mathbb{Z}$, define the N=2 superconformal map
\begin{eqnarray}\label{define-I}
I_n:   \mbox{$(\bigwedge_{*>1}^0)^\times$} \oplus (\mbox{$\bigwedge_{*>1}^1$})^2  & \longrightarrow &   \mbox{$(\bigwedge_{*>1}^0)^\times$} \oplus (\mbox{$\bigwedge_{*>1}^1$})^2 \\
(z, \theta^+, \theta^-) & \mapsto & I_n(z,\theta^+, \theta^-) =  \Bigl(\frac{1}{z}, \; \frac{i\theta^+z^n}{z}, \; \frac{i\theta^-z^{-n}}{z} \Bigr) . \nonumber
\end{eqnarray}

For $n \in \mathbb{Z}$, define $S^2\hat{\mathbb{C}}(n)$ to be the genus-zero N=2 superconformal super-Riemann surface over $\bigwedge_{*>1}$ with N=2 superconformal structure given by the covering of local coordinate neighborhoods $\{ U_{\sou_n}, U_{\nor_n} \}$ and the local coordinate maps
\begin{eqnarray}
\sou_n  : U_{\sou_n}  & \longrightarrow & \mbox{$\bigwedge_{*>1}^0$} \oplus (\mbox{$\bigwedge_{*>1}^1$})^2 \label{southern-chart} \\
\nor_n : U_{\nor_n} & \longrightarrow & \mbox{$\bigwedge_{*>1}^0$}  \oplus (\mbox{$\bigwedge_{*>1}^1$})^2, \label{northern-chart}
\end{eqnarray}
which are homeomorphisms of $U_{\sou_n}$ and $U_{\nor_n}$ onto $\bigwedge_{*>1}^0 \oplus (\bigwedge_{*>1}^1)^2$, respectively, such that 
\begin{eqnarray}\label{transition}
\sou_n \circ \nor_n^{-1} : \mbox{$(\bigwedge_{*>1}^0)^\times$} \oplus \mbox{$(\bigwedge_{*>1}^1)^2$} &\longrightarrow& \mbox{$(\bigwedge_{*>1}^0)^\times$} \oplus \mbox{$(\bigwedge_{*>1}^1)^2$} \label{origin-of-I}\\ 
(z, \theta^+,\theta^-) & \mapsto & I_n(z,\theta^+,\theta^-) . \nonumber
\end{eqnarray}
Thus the body of $S^2\hat{\mathbb{C}}(n)$ is the Riemann sphere, i.e., $(S^2\hat{\mathbb{C}}(n))_B = \hat{\mathbb{C}} = \mathbb{C} \cup \{\infty\}$.

The N=1 superanalytic super-Riemann surface $\mathcal{F}(S^2\hat{\mathbb{C}}(n))$, for $n \in \mathbb{Z}$, has body $\hat{\mathbb{C}}$, and transition function given by $\mathcal{F}_1 (I_n) (z, \theta) = (1/z, i\theta z^{n-1})$.  

In \cite{B-uniformization}, we proved the following Uniformization Theorem:

\begin{thm}\label{uniformization-thm} (\cite{B-uniformization}) Any N=2 superconformal super-Riemann surface with closed, genus-zero body is N=2 superconformally equivalent to $S^2\hat{\mathbb{C}} (n)$ for some $n \in \mathbb{Z}$.  Moreover, $S^2\hat{\mathbb{C}} (m)$ and  $S^2\hat{\mathbb{C}} (n)$ for $m,n \in \mathbb{Z}$ are not N=2 superconformally equivalent if $m \neq n$.  

Similarly, any N=1 superanalytic super-Riemann surface with closed, genus-zero body is N=1 superanalytically equivalent to $\mathcal{F}(S^2\hat{\mathbb{C}} (n))$ for some $n \in \mathbb{Z}$.  Moreover, $\mathcal{F}(S^2\hat{\mathbb{C}} (m))$ and  $\mathcal{F}(S^2\hat{\mathbb{C}} (n))$ for $m,n \in \mathbb{Z}$ are not N=1 superanalytically equivalent if $m \neq n$.  
\end{thm}

\begin{rema}\label{line-bundle-remark}
{\em As discussed in \cite{B-uniformization}, this Uniformization Theorem for genus-zero N=2 superconformal (resp. N=1 superanalytic) super-Riemann surfaces, Theorem \ref{uniformization-thm},  can be restated as follows:  There is a bijection between N=2 superconformal (resp. N=1 superanalytic) equivalence classes of N=2 superconformal (resp. N=1 superanalytic) super-Riemann surfaces with closed, genus-zero body and holomorphic equivalence classes of holomorphic line bundles over the Riemann sphere.  One can see this bijective correspondence explicitly, by noting that the N=2 superconformal super-Riemann sphere $S^2 \hat{\mathbb{C}}(n)$ for $n \in \mathbb{Z}$ has, as a substructure, the $GL(1, \mathbb{C})$-bundle over $\hat{\mathbb{C}}$ given by the transition function $iz_{(\emptyset)}^{n-1} : \mathbb{C}^\times \longrightarrow \mathbb{C}^\times$, corresponding to the transition function for the first fermionic component of $S^2\hat{\mathbb{C}}(n)$ restricted to the fiber in the first component of $\theta^+ = \theta_{(1)}^+ \zeta_1 + \theta_{(2)}^+ \zeta_2 + \cdots$.  (Or equivalently, one can restrict to the $(j)$-th component for $(j) \in J^1_{*>1}$.)  Moreover, the $GL(1, \mathbb{C})$-bundle over $\hat{\mathbb{C}}$ with transition function $iz_{(\emptyset)}^{n-1} : \mathbb{C}^\times \longrightarrow \mathbb{C}^\times$, for $n \in \mathbb{Z}$, picks out a unique $S^2 \hat{\mathbb{C}}(n)$.    Under this bijection between equivalence classes of genus-zero  N=2 superconformal super-Riemann surfaces and equivalence classes of holomorphic line bundles over the body, the N=2 superconformal super-Riemann surface $S^2\hat{\mathbb{C}} (n)$ corresponds to the holomorphic line bundle over $\hat{\mathbb{C}}$ of degree $-n+1$.    }
\end{rema}

\begin{rema}\label{H-manifolds-remark} {\em
In particular, in the genus-zero case, any N=2 superconformal (resp. N=1 superanalytic) super-Riemann surface is N=2 superconformally (resp. N=1 superanalytically) equivalent to an $\mathcal{H}$-manifold, and thus the classification of such supermanifolds reduces to that found in the ringed-space approach as in \cite{Manin1}.  This is not the case for genus greater than one.  As shown in \cite{B-uniformization}, whether or not a general N=2 superconformal (resp. N=1 superanalytic) super-Riemann surface $M$  is  N=2 superconformally (resp. N=1 superanalytically) equivalent to an $\mathcal{H}$-manifold is dependent on whether the first \v Cech cohomology group of the underlying body Riemann surface $M_B$ with coefficients in the sheaf of holomorphic vector fields over the body is trivial.   }
\end{rema}

Throughout the remainder of this paper, we will use the setting of N=2 superconformal super-Riemann spheres and their automorphism groups.  However, using Proposition \ref{DRS-prop}, the category isomorphism (\ref{category-iso}), and the map (\ref{N=2-to-N=1}), our results can easily be translated to the N=1 superanalytic setting. 

\section{The automorphism groups of the N=2 superconformal super-Riemann spheres}

Let  $\mathrm{Aut}( S^2\hat{\mathbb{C}} (n))$ denote the group of N=2 superconformal automorphisms of $S^2\hat{\mathbb{C}} (n)$, for $n \in \mathbb{Z}$.  For $T \in \mathrm{Aut}( S^2\hat{\mathbb{C}} (n))$, define the N=2 superconformal function $T_\sou = \sou_n \circ T \circ \sou_n^{-1}$ where $(U_{\sou_n}, \sou_n)$ is the coordinate chart (\ref{southern-chart}) defining $S^2\hat{\mathbb{C}}(n)$.   Then by (\ref{superconformal-condition1}) and (\ref{superconformal-condition3}), $T_\sou$ is uniquely determined by three even superfunctions of one even variable $f(z)$ and $g^\pm(z)$, and two odd superfunctions of one even variable $\psi^\pm(z)$ satisfying the condition (\ref{superconformal-condition4}).  We will call these 5 functions the {\it component} functions of $T_\sou$.  

\begin{thm}\label{auto-groups-thm} 
If $T \in \mathrm{Aut}(S^2\hat{\mathbb{C}}(n))$, then $T$ is uniquely determined by $T_\sou = \sou_n \circ T \circ \sou_n^{-1}$ with the $T_\sou$ determined by component functions as follows:
\begin{equation}\label{f}
f(z) = \frac{az + b}{cz + d} \qquad \mbox{for $a,b,c,d \in \bigwedge_{*-2}^0$ and $ad-bc = 1$.}
\end{equation}
If $n=0$, we have
\begin{equation}\label{psi-zero}
\psi^\pm(z)  = \frac{\psi^\pm_1z + \psi^\pm_0}{cz+d} \qquad \mbox{for $\psi^\pm_j \in \bigwedge^1_{*-2}$ with $j = 0,1$},
\end{equation}
and
\begin{equation}
g^\pm (z) = \frac{\epsilon^\pm}{cz+d} + \frac{f^\pm z + h^\pm}{(cz+d)^2}
\end{equation}
for $\epsilon^\pm \in  (\bigwedge^0_{*-2})^\times$, $f^\pm, h^\pm \in \bigwedge^0_{*-2}$ satisfying
\begin{eqnarray}
\epsilon^+ \epsilon^- &=& 1 - \psi_1^+ \psi^-_0 - \psi^-_1 \psi^+_0\\
f^\pm &=& \mp \epsilon^\pm \psi^+_1 \psi^-_1 d\\
h^\pm &=& \pm \epsilon^\pm ( \psi^+_0 \psi^-_0 c - (\psi^+_1 \psi^-_0 - \psi^-_1 \psi^+_0)d \mp \psi^+_1 \psi^-_1 \psi^+_0 \psi^-_0 d).
\end{eqnarray}

If $n = 1$, we have
\begin{eqnarray}
\psi^+(z) & =& \psi^+_0 \label{psi-+1+}\\
\psi^-(z) &=& \frac{\psi^-_2 z^2 + \psi^-_1z + \psi^-_0}{(cz+d)^2} \label{psi-+1-}\\
g^+ (z) &=& \epsilon \\
g^-(z) &=& \frac{1}{\epsilon (cz+d)^2} + \frac{\psi^+_0(2 \psi^-_2 dz - \psi^-_1 (cz-d) - 2\psi^-_0c)}{\epsilon (cz+d)^3}, \label{g-+1-}
\end{eqnarray}
for $\psi^+_0, \psi^-_j \in \bigwedge^1_{*-2}$ with $j = 0,1,2$, and $\epsilon \in (\bigwedge^0_{*-2})^\times$.

If $n = -1$, we have
\begin{eqnarray}
\psi^+(z) &=& \frac{\psi^+_2 z^2 + \psi^+_1z + \psi^+_0}{(cz+d)^2} \label{psi--1+} \\
\psi^-(z) & =& \psi^-_0 \label{psi--1-}\\
g^+(z) &=& \frac{1}{\epsilon (cz+d)^2} - \frac{(2 \psi^+_2 dz - \psi^+_1 (cz-d) - 2\psi^+_0c)\psi^-_0}{\epsilon (cz+d)^3},\\
g^- (z) &=& \epsilon ,
\end{eqnarray}
for $\psi^+_j, \psi^-_0\in \bigwedge^1_{*-2}$ with $j = 0,1,2$, and $\epsilon \in (\bigwedge^0_{*-2})^\times$.

If $n\geq 2$, we have
\begin{eqnarray}
\psi^+(z) & =& 0 \label{psi-positive-zero}\\
\psi^-(z) &=& \frac{\psi^-_{n+1} z^{n+1} + \psi^-_nz^n + \cdots  + \psi^-_1z + \psi^-_0}{(cz+d)^{n+1}} \label{psi-positive}\\
g^\pm (z) &=& \epsilon^{\pm1} (cz+d)^{\pm n-1} , \label{g-positive}
\end{eqnarray}
for $\psi^-_j \in \bigwedge^1_{*-2}$ with $j = 0,1,\dots, n+1$, and $\epsilon \in (\bigwedge^0_{*-2})^\times$.

If $n\leq -2$, we have
\begin{eqnarray}
\psi^+(z) &=& \frac{\psi^+_{-n+1} z^{-n+1} + \psi^+_{-n}z^{-n} + \cdots  + \psi^+_1z + \psi^+_0}{(cz+d)^{-n+1}} \label{psi-negative}\\
\psi^-(z) & =& 0  \label{psi-negative-zero}\\
g^\pm (z) &=& \epsilon^{\mp1} (cz+d)^{\pm n-1} , \label{g-negative}
\end{eqnarray}
for $\psi^-_j \in \bigwedge^1_{*-2}$ with $j = 0,1,\dots, n+1$, and $\epsilon \in (\bigwedge^0_{*-2})^\times$.
\end{thm}

\begin{proof}
For any $T \in \mathrm{Aut} (S^2\hat{\mathbb{C}}(n))$, we have that $T$ restricted to the body of $S^2 \hat{\mathbb{C}}(n)$ is an automorphism of the Riemann sphere.  Thus  
\begin{eqnarray}\label{T-body}
(T_\sou)_B(z, \theta^+, \theta^-) &=& \pi_B^{(1,2)} \circ T_\sou (z, \theta^+, \theta^-) 
= f_B(z_B) = \pi_B^{(1,0)} \circ f(z)\\
& =&  \frac{a_Bz_B + b_B}{c_B z_B + d_B} \nonumber
\end{eqnarray} 
for $a_B, b_B, c_B, d_B \in \mathbb{C}$ satisfying $a_B d_B - b_B c_B = 1$.   Therefore, the only possible singularity for $T_\sou$ is at $z_B = -d_B/c_B$.   In addition, defining $T_\nor = \nor_n \circ T \circ \nor_n^{-1}$; that is, $T_\nor = I_n^{-1} \circ T_\sou \circ I_n$, then we have that the only possible singularity for $T_\nor$ is at $z_B = -a_B/b_B$.   

Let $\tilde{f}(z)$, $\tilde{g}^\pm(z)$, and $\tilde{\psi}^\pm$ be the three even and two odd component functions that determine $T_\nor(z, \theta^+, \theta^-)$.  Then in terms of the component functions $f$, $g^\pm$ and $\psi^\pm$ that determine $T_\sou$, we have
\begin{eqnarray}
\tilde{f}(z) \! \! \! &=& \! \! \! \frac{1}{f(1/z)}\\
\tilde{\psi}^\pm(z) \! \! \! &=&\! \! \!  - i \psi^\pm(1/z) (f(1/z))^{\pm n - 1} \label{psi-tilde} \\
\qquad \ \ \ \tilde{g}^\pm (z) \! \! \! &=& \! \! \! z^{\pm n -1} g^\pm (1/z) (f(1/z))^{\pm n - 2} \left( f(1/z)   - (n \mp 1) \psi^+(1/z) \psi^-(1/z) \right). \label{g-tilde}
\end{eqnarray}

Expanding $f(z)$ about the pole at $z_B = -d_B/c_B$ and taking into consideration that $\tilde{f}(z)$ can only have poles at $z_B = -a_B/b_B$, and that $f_B(z_B)$ is given by (\ref{T-body}), we have that there exists $a,b,c,d \in \bigwedge_{*-2}^0$ such that 
\begin{equation}
f(z) = \frac{az+b}{cz+d}
\end{equation}
where $a_B d_B - b_B c_B = 1$.  Furthermore, if $ad-bc \neq 1$, then we can normalize $a,b,c,d$ by dividing by $\sqrt{ad-bc}$, so that $f(z) = (a'z+ b')/(c'z + d')$ for $a', b', c', d' \in \bigwedge_{*-2}^0$ with $a'd'-b'c' = 1$.   This proves (\ref{f}).

Now, requiring that $\psi^\pm(z)$ only have poles at $z = -d/c$,  Eq. (\ref{psi-tilde}) and the restriction that $\tilde{\psi}^\pm$ only have poles at $z = -a/b$, implies that the $\psi^\pm$ must be of the form as stated in the theorem; that is, Eqs. (\ref{psi-zero}), (\ref{psi-+1+}), (\ref{psi-+1-}), (\ref{psi--1+}), (\ref{psi--1-}), (\ref{psi-positive-zero}), (\ref{psi-positive}), (\ref{psi-negative}),  and (\ref{psi-negative-zero}) hold. 

Let $n \geq 2$.  Requiring that $g^-(z)$ only have poles at $z = -d/c$,  Eq. (\ref{g-tilde}) and the restriction that $\tilde{g}^-(z)$ only have poles at $z = -a/b$, implies that $g^-(z) = \epsilon (cz+d)^{-n-1}$ for some constant $\epsilon \in \bigwedge_{*-2}^0$.  The N=2 superconformal condition (\ref{superconformal-condition4}) applied to $T_\sou$, then implies that $g^+(z) = (g^-(z))^{-1} (cz+d)^{-2}$ giving (\ref{g-positive}).  Eq. (\ref{g-negative}) is obtained analogously.

Now let $n = 1$.  Requiring that $g^+(z)$ only have poles at $z = -d/c$,  Eq. (\ref{g-tilde}) and the restriction that $\tilde{g}^+$ only have poles at $z = -a/b$, implies that $g^+(z) = \epsilon$ for some constant $\epsilon \in \bigwedge_{*-2}^0$.  The N=2 superconformal condition (\ref{superconformal-condition4}) applied to $T_\sou$, then implies that $g^-(z)$ must be of the form (\ref{g-+1-}).  Then one must check the condition that $\tilde{g}^-(z)$ can only have poles at $z = -a/b$, which is indeed satisfied.      The case for $n = -1$ is proved analogously.

The case $n = 0$ was proved in \cite{B-n2moduli}.

It remains to show that $T$ can be uniquely determined, by $T_\sou$.  We have 
\begin{multline}
T_\sou : \bigl(\mbox{$\bigwedge_{*>1}^0$} \smallsetminus \bigl(\{- d_B/c_B \} \times (\mbox{$\bigwedge_{*>1}^0$})_S \bigr) \bigr) \oplus\mbox{$(\bigwedge_{*>1}^1)^2$}  \longrightarrow \\
\bigl( \mbox{$\bigwedge_{*>1}^0$} \smallsetminus \bigl(\{a_B/c_B \} \times
(\mbox{$\bigwedge_{*>1}^0$})_S\bigr) \bigr) \oplus \mbox{$(\bigwedge_{*>1}^1)^2$},
\end{multline}
and
\begin{multline}
T_\nor: \bigl(\mbox{$\bigwedge_{*>1}^0$} \smallsetminus \bigl(\{-a_B/b_B \} \times
(\mbox{$\bigwedge_{*>1}^0$})_S \bigr)\bigr) \oplus \mbox{$(\bigwedge_{*>1}^1)^2$} \longrightarrow
\\
\bigl( \mbox{$\bigwedge_{*>1}^0$} \smallsetminus \bigr(\{d_B/b_B \} \times
(\mbox{$\bigwedge_{*>1}^0$})_S \bigr) \bigr) \oplus \mbox{$(\bigwedge_{*>1}^1)^2$},
\end{multline}
where $T_\nor = I_n^{-1} \circ T_\sou \circ I_n$ on the restricted domain.

Thus
\begin{equation}\label{T1}
T(p) = \left\{
  \begin{array}{ll} 
      \sou_n^{-1} \circ T_\sou \circ \sou_n (p) & \mbox{if $p \in U_{\sou_n} \smallsetminus X_1$}, \\  
      \nor_n^{-1} \circ T_\nor \circ \nor_n (p) & \mbox{if $p \in U_{\nor_n} \smallsetminus X_2$},
\end{array} \right.
\end{equation}
where $X_1 = \sou_n^{-1}( (\{- d_B/c_B \} \times (\bigwedge_{*>1}^0)_S ) \oplus (\bigwedge_{*>1}^1)^2 )$ and $X_2 = \nor_n^{-1}(( \{-a_B/b_B \} \times (\bigwedge_{*>1}^0)_S ) \oplus  (\bigwedge_{*>1}^1)^2 )$.
This defines $T$ for all $p \in S^2\hat{\mathbb{C}}(n)$ unless: 

(i) $a_B = 0$ and  $p \in \nor_n^{-1}( (\{0 \} \times (\bigwedge_{*>1}^0)_S) \oplus  \mbox{$(\bigwedge_{*>1}^1)^2$}
)$; or 

(ii) $d_B = 0$ and $p \in \sou_n^{-1}((\{0 \} \times (\bigwedge_{*>1}^0)_S) \oplus
\mbox{$(\bigwedge_{*>1}^1)^2$})$.

In case (i), we define $T (p)$ as follows:
Let $H(z, \theta^+, \theta^-) = T_\sou \circ I_n(z, \theta^+, \theta^-)$.  Then the domain of $H$ extends to $(z_S, \theta^+, \theta^-) \in (\{0\} \times (\bigwedge^0_*)_S \oplus (\bigwedge_*^1)^2)$.  Define $T(p) = \sou_n^{-1} \circ H(z, \theta^+ \theta^-)$, for $\nor_n (p) = (z,\theta^+, \theta^-) = (z_S, \theta^+, \theta^-)$.

In case (ii), we define $T(p)$ as follows:
Let $H(z, \theta^+, \theta^-) = I_n^{-1} \circ T_\nor (z, \theta^+, \theta^-)$.  Then the domain of $H$ extends to $(z_S, \theta^+, \theta^-) \in (\{0\} \times (\bigwedge^0_*)_S \oplus (\bigwedge_*^1)^2)$.  Define $T(p)  = \nor_n^{-1} \circ H(z, \theta^+, \theta^-)$,  for $\sou_n (p) = (z,\theta^+, \theta^-) = (z_S, \theta^+, \theta^-)$.

Thus $T$ is uniquely determined by $T_\sou$, i.e., by its value on $\sou_n (U_{\sou_n})$. 
\end{proof}

\section{The Lie superalgebras of infinitesimal global transformations of N=2 super-Riemann spheres}

The {\it N=2 Neveu-Schwarz algebra} is the Lie superalgebra with basis consisting of the central element $d$, even elements $L_n$ and $J_n$, and odd elements $G^\pm_{n + 1/2}$ for $n \in \mathbb{Z}$, and commutation relations  
\begin{eqnarray}
\left[L_m ,L_n \right] \! \! &=& \! \! (m - n)L_{m + n} + \frac{1}{12} (m^3 - m) \delta_{m + n 
, 0} \; d , \label{Virasoro-relation} \\
\left[ J_m, J_n \right] \! \! &=& \! \! \frac{1}{3} m \delta_{m+n,0} d, \qquad \qquad \qquad \qquad \quad \ \ \left [L_m, J_n \right] \  = \ -n J_{m+n}, \label{Virasoro-and-U(1)}\\
\left[ L_m, G^\pm_{n + \frac{1}{2}} \right] \! \! &=& \! \! \left(\frac{m}{2} - n - \frac{1}{2} \right) G^\pm_{m+n+\frac{1}{2}} ,\\
\left[ J_m, G^\pm_{n + \frac{1}{2}} \right] \! \! &=& \! \!  \pm G^\pm_{m+n+\frac{1}{2}} , \qquad \qquad \qquad \ \ \ \left[ G^\pm_{m + \frac{1}{2}} , G^\pm_{n + \frac{1}{2}} \right] \ =\  0 , \label{J-relation}\\
\qquad \left[ G^+_{m + \frac{1}{2}} , G^-_{n - \frac{1}{2}} \right]  \! \! &=& \! \! 2L_{m + n} + (m-n+1) J_{m+n}  + \frac{1}{3} (m^2 + m) \delta_{m + n , 0} \; d ,\label{Neveu-Schwarz-relation-last} 
\end{eqnarray}
for $m, n \in \mathbb{Z}$. 

Let $x$ denote a formal even variable, and $\varphi^+$ and $\varphi^-$ denote odd (i.e., anticommuting) variables.   Consider the superderivations in $\mbox{Der} (\bigwedge_*[[x,x^{-1}]] [\varphi^+,\varphi^-])$, given by 
\begin{eqnarray}
L_n(x,\varphi^+,\varphi^-) &=& - \biggl( x^{n + 1} \frac{\partial}{\partial x} + \Bigl(\frac{n + 1}{2}\Bigr) x^n \Bigl( \varphi^+ \frac{\partial}{\partial \varphi^+} + \varphi^- \frac{\partial}{\partial \varphi^-}\Bigr) \biggr) \label{L-notation}\\
J_n(x,\varphi^+,\varphi^-) &=& - x^n\Bigl(\varphi^+\frac{\partial}{\partial \varphi^+} - \varphi^- \frac{\partial}{\partial \varphi^-}\Bigr)  \label{J-notation}\\
\quad G^\pm_{n -\frac{1}{2}} (x,\varphi^+,\varphi^-) &=& - \biggl( x^n \Bigl( \frac{\partial}{\partial \varphi^\pm} - \varphi^\mp \frac{\partial}{\partial x}\Bigr) \pm nx^{n-1} \varphi^+ \varphi^- \frac{\partial}{\partial \varphi^\pm} \biggr) \label{G-notation}
\end{eqnarray}
for $n \in \mathbb{Z}$.  These give a representation of the N=2 Neveu-Schwarz algebra with central charge zero; that is  (\ref{L-notation})--(\ref{G-notation}) satisfy (\ref{Virasoro-relation})--(\ref{Neveu-Schwarz-relation-last}) with $d = 0$.  In addition, in \cite{B-n2moduli} it is shown that with this representation, this is the Lie superalgebra of infinitesimal N=2 superconformal transformations. 

Let $y$ be an even formal variable and $\xi$ an odd formal variable.  By direct expansion, we have that 
\begin{eqnarray}
e^{-yL_{-1}(x, \varphi^+, \varphi^-)} \cdot (x,\varphi^+,\varphi^-) &=&  (x + y, \, \varphi^+, \, \varphi^-)  \label{group-action1} \\
e^{-yL_0(x, \varphi^+, \varphi^-)} \cdot (x,\varphi^+,\varphi^-) &=&  (e^y x,\, e^{\frac{y}{2}}\varphi^+, \, e^{\frac{y}{2}} \varphi^-) \\
e^{-yJ_0(x, \varphi^+, \varphi^-)} \cdot (x,\varphi^+,\varphi^-) &=&  (x, \, e^y \varphi^+, \, e^{-y} \varphi^-) \\
\quad e^{-\xi G^+_{-1/2}(x, \varphi^+, \varphi^-)} \cdot (x,\varphi^+,\varphi^-) &=& (x + \varphi^-  \xi , \, \xi + \varphi^+, \, \varphi^-) \\ 
e^{-\xi G^-_{-1/2}(x, \varphi^+, \varphi^-)} \cdot (x,\varphi^+,\varphi^-)&=& (x + \varphi^+ \xi, \, \varphi^+, \, \xi + \varphi^-) ,
\end{eqnarray}
and for $k \in \mathbb{Z}_+$
\begin{eqnarray}
e^{-\xi G^+_{k-1/2}(x, \varphi^+, \varphi^-)} \cdot (x,\varphi^+,\varphi^-) &=& (x + \varphi^- \xi x^k, \, \xi x^k + \varphi^+ \\
& & \quad + \, k\varphi^+ \varphi^- \xi x^{k-1} ,  \, \varphi^-) \nonumber \\
e^{-\xi G^-_{k-1/2}(x, \varphi^+, \varphi^-)} \cdot (x,\varphi^+,\varphi^-) &=& (x + \varphi^+ \xi x^k, \, \varphi^+, \, \xi x^k +  \varphi^- \\
& & \quad - \, k\varphi^+ \varphi^- \xi x^{k-1}). \nonumber
\end{eqnarray}
For $n \in \mathbb{Z}$, we have 
\begin{eqnarray*}
\lefteqn{e^{-y(L_1(x, \varphi^+, \varphi^-) - n J_1(x, \varphi^+, \varphi^-)} \cdot (x,\varphi^+,\varphi^-)}\\
&=& \exp \left( y \left( x^2 \frac{\partial}{\partial x} + (1-n)x \varphi^+ \frac{\partial}{\partial \varphi^+}  + (1+n) x \varphi^- \frac{\partial}{\partial \varphi^-} \right) \right) \cdot (x, \varphi^+, \varphi^-)\\
&=& \left( \sum_{k \in \mathbb{N}} y^k x^{k+1} , \, \varphi^+ \sum_{k \in \mathbb{N}} \binom{n-1}{k} (-1)^k y^k x^k, \, \varphi^- \sum_{k \in \mathbb{N}}  \binom{-n-1}{k} (-1)^k y^kx^k \right)\\
&=&  \left(\frac{x}{1-yx},  \, \frac{\varphi^+}{(1-yx)^{-n+1}}, \, \frac{\varphi^-}{(1-yx)^{n+1}}\right).\\
\end{eqnarray*}
In addition, we note that
\begin{equation}\label{basis-shift}
e^{-y(L_0(x, \varphi^+, \varphi^-) - \frac{n}{2} J_0(x, \varphi^+, \varphi^-)} \cdot (x,\varphi^+,\varphi^-) =  (e^y x,\, e^{\frac{y(1-n)}{2}}\varphi^+, \, e^{\frac{y(1+n)}{2}} \varphi^-).
\end{equation}
Thus letting $a = d^{-1}= e^{y/2}$, Eq. (\ref{basis-shift}) gives $((a/d) x, \varphi^+ d^{n-1}, \varphi^- d^{-n-1})$, with $ad = 1$.

This along with Theorem \ref{auto-groups-thm} implies the following:

\begin{thm}\label{infinitesimals-thm}
For $n \in \mathbb{Z}$, let $\mathfrak{g}_n$ be the Lie superalgebra of infinitesimal automorphisms of $S^2\hat{\mathbb{C}}(n)$, i.e., $\mathfrak{g}_n = Lie(\mathrm{Aut}(S^2\hat{\mathbb{C}} (n)))$.  Then each $\mathfrak{g}_n$ is a subalgebra of the N=2 Neveu-Schwarz algebra, and these subalgebras are given as follows
\begin{eqnarray}
\mathfrak{g}_0 \! \! &=& \! \! \mathrm{span}_\mathbb{C} \left\{ L_{-1}, \, L_0,  \, L_1, \, J_0, \, G^+_{-1/2}, \, G^+_{1/2}, \, G^-_{-1/2}, \, G^-_{1/2}\right\}\\
\mathfrak{g}_1 \! \! &=& \! \! \mathrm{span}_\mathbb{C} \left\{L_{-1}, \,  L_0 - \frac{1}{2} J_0,  \, L_1-J_1, \, J_0, \, G^+_{-1/2}, \, G^-_{-1/2}, \, G^-_{1/2}, \, G^-_{3/2}\right\}\\
\qquad \ \mathfrak{g}_{-1} \! \! &=&\! \!  \mathrm{span}_\mathbb{C} \left\{ L_{-1}, \, L_0 + \frac{1}{2}J_0, \, L_1+J_1, \, J_0, \, G^+_{-1/2}, \, G^+_{1/2}, \, G^+_{3/2}, \, G^-_{-1/2}\right\},\\
\mathfrak{g}_n \! \! &=& \! \! \mathrm{span}_\mathbb{C} \left\{ L_{-1}, \, L_0 - \frac{n}{2}J_0 , \, L_1-nJ_1, \, J_0, \, G^-_{-1/2}, \, G^-_{1/2}, \dots, \right. \\
& & \hspace{2.75in} \left. G^-_{n-1/2}, \, G^-_{n+1/2}\right\}, \nonumber
\end{eqnarray}
for $n\geq 2$, and
\begin{multline}
\mathfrak{g}_n \ = \ \mathrm{span}_\mathbb{C} \left\{ L_{-1}, \, L_0 - \frac{n}{2} J_0, \, L_1-nJ_1, \, J_0, \, G^+_{-1/2}, \, G^+_{1/2}, \dots,\right.  \\
\left. G^+_{-n-1/2}, \, G^+_{-n+1/2}\right\},
\end{multline}
 for $n\leq -2$.
\end{thm}

\begin{rema} {\em 
From Remark \ref{line-bundle-remark}, we see that the integer $n$ is related to the degree of the line bundle over $\hat{\mathbb{C}}$ which is canonically determined by $S^2\hat{\mathbb{C}}(n)$.  Considering the coordinate transition function (\ref{define-I}),  the integer $n$ is also a measure of the asymmetry between the fermionic (odd) components of $S^2\hat{\mathbb{C}}(n)$.  Thus the dimension of the Lie algebra of infinitesimal global N=2 superconformal transformations grows, roughly, as the absolute value of the degree of the corresponding line bundle or, equivalently, the degree of fermionic asymmetry.   That is $\mathrm{dim} \, \mathfrak{g}_n = 4+4$ if $|n| \leq 2$ and $\mathrm{dim} \, \mathfrak{g}_n = 4+(|n| + 2)$ if $|n| \geq 2$.
}
\end{rema}

The reason we have written the basis in this rather unusual way --- with $L_0 - \frac{n}{2} J_0$ and $J_0$ rather than $L_0$ and $J_0$ ---  is due to the more natural correspondence with the action of the group of automorphisms, cf. Eq. (\ref{basis-shift}) above, as well as Eq. (\ref{even-part-of-g_n}) below.

For $n \in \mathbb{Z}$, let $\mathfrak{g}_n = \mathfrak{g}_n^0 \oplus \mathfrak{g}_n^1$ denote the decomposition of the Lie superalgebra into even and odd components.  Then $\mathfrak{g}_n^0$ is a Lie algebra and decomposes into the direct sum of Lie algebras as follows:
\begin{eqnarray}
\mathfrak{g}_n^0 &=& \mathrm{span}_\mathbb{C} \{ L_{-1}, \, L_0 - \frac{n}{2} J_0,  \, L_{-1}, \, L_1 - nJ_1, \, J_0\} \label{even-part-of-g_n}\\
&=& \mathrm{span}_\mathbb{C} \{  L_{-1}, \, L_0 - \frac{n}{2} J_0,  \, L_{-1}, \, L_1 - nJ_1\} \oplus \mathbb{C} J_0 \nonumber \\
&\cong& \mathfrak{sl} (2, \mathbb{C}) \oplus \mathfrak{gl}(1, \mathbb{C}), \nonumber
\end{eqnarray}
for each $n \in \mathbb{Z}$.  The isomorphism with  $\mathfrak{sl}(2, \mathbb{C})$ is given explicitly by

\begin{eqnarray}\label{sl2}
\left(\begin{array}{cc}  
0 & 1\\
0 & 0 \end{array} \right)  \ \longleftrightarrow   \  L_{-1},  & & 
\frac{1}{2} \left( \begin{array}{cc}  
1 & 0  \\
0 & -1   \end{array} \right)    \ \longleftrightarrow \  L_0 - \frac{n}{2} J_0, \\
\left(\begin{array}{cc}   
0 & 0  \\
-1 & 0 \end{array} \right)  & \longleftrightarrow  & L_1 - nJ_1.  \nonumber
\end{eqnarray}

In the case of $n = 0$, the Lie superalgebra $\mathfrak{g}_0$ is isomorphic to the orthogonal-symplectic superalgebra $\mathfrak{osp} (2|2, \mathbb{C})$; see \cite{B-n2moduli}.  This is also denoted $C(2)$ in \cite{Kac}, and is in the family of simple Lie superalgebras $C(m)$.  Explicitly, we have  
\begin{eqnarray*}
\mathfrak{osp}(2|2, \mathbb{C}) = \left\{ \biggl. \left(\begin{array}{cc|cc}   
                                         a & b & s & q \\
                                         c & -a & -r & -p \\
                                         \hline
                                         p & q  & d & 0 \\
                                         r & s & 0 & -d 
                                            \end{array} \right) \in
\mathfrak{gl}(2|2, \mathbb{C}) \; \biggr| \; a,b,c,d,p,q, r, s\in \mathbb{C} \right\} 
\end{eqnarray*}
which is the subalgebra of $\mathfrak{gl}(2|2, \mathbb{C})$ leaving a certain non-degenerate form invariant (cf. \cite{Kac}, \cite{B-n2moduli}).  The isomorphism $\mathfrak{g}_0 \cong \mathfrak{osp} (2|2, \mathbb{C})$ is given explicitly, for instance, as follows: $L_{-1}$, $L_0$, and $L_1$  are given by the embedding of $\mathfrak{sl}(2, \mathbb{C})$ into $\mathfrak{osp}(2|2, \mathbb{C})$ via
\begin{equation}
\left( \begin{array}{cc}
a & b \\
c & -a
\end{array}\right)
\longrightarrow
\left( \begin{array}{cc|cc}
a & b & 0 &0 \\
c & -a & 0 & 0\\
\hline
0 & 0 & 0 & 0 \\
0 & 0 & 0 & 0
\end{array}\right),
\end{equation}
and the isomorphism defined by (\ref{sl2}).   Then for $J_0$ we have
\begin{equation}
\left(\begin{array}{cc|cc}   
0 & 0 & 0 & 0 \\
0 & 0 & 0 & 0 \\
\hline
0 & 0 & 1 & 0 \\
0 & 0 & 0 & -1 \end{array} \right) \ \ \longleftrightarrow \ \ J_0, \\ 
\end{equation}
and for the odd components, we have  
\begin{eqnarray}
\begin{array}{ccccccc}
\quad & \left(\begin{array}{cc|cc}
0 & 0 & 0 & 1 \\   
0 & 0 & 0 & 0 \\
\hline
0 & 1 & 0 & 0 \\
0 & 0 & 0 & 0 \end{array} \right)  &\longleftrightarrow& G^+_{-\frac{1}{2}} , & \qquad
\left(\begin{array}{cc|cc}    
0 & 0 & 0 & 0 \\
0 & 0 & 0 & -1 \\
\hline
1 & 0 & 0 & 0 \\
0 & 0 & 0 & 0 \end{array} \right) &\longleftrightarrow& G^+_{\frac{1}{2}} ,  \\
\\
\quad  & \left(\begin{array}{cc|cc}
0 & 0 & 1 & 0 \\   
0 & 0 & 0 & 0 \\
\hline
0 & 0 & 0 & 0 \\
0 & 1 & 0 & 0 \end{array} \right) & \longleftrightarrow& G^-_{-\frac{1}{2}} , & \qquad 
\left(\begin{array}{cc|cc}    
0 & 0 & 0 & 0 \\
0 & 0 & -1 & 0 \\
\hline
0 & 0 & 0 & 0 \\
1 & 0 & 0 & 0 \end{array} \right) &\longleftrightarrow&  G^-_{\frac{1}{2}}.  
\end{array}
\end{eqnarray}

We next address the cases $n = \pm1$.   Let $\mathfrak{p}(2|2, \mathbb{C})$ denote the seven-dimensional Lie superalgebra which is the subalgebra of $\mathfrak{sl}(2|2, \mathbb{C})$ given by
\begin{equation}
\mathfrak{p}(2|2, \mathbb{C}) = \left\{ \left( \begin{array}{cc|cc} a & b & p &q\\
c & -a & q & r\\
\hline 
0 & s & -a & -c \\
-s & 0 & -b & a
\end{array}
\right) \; | \; a,b,c,p,q, r,s  \in \mathbb{C} \right\}.
\end{equation}
This is the simple Lie superalgebra denoted $P(1)$ in \cite{Kac} which is in the series $P(m)$ of simple Lie superalgebras.  Then $\mathfrak{g}_{\pm 1}$ is the semi-direct product of $\mathfrak{gl}(1, \mathbb{C}) \cong \mathbb{C} J_0$ with $\mathfrak{p}(2|2, \mathbb{C})$.  Explicitly, the isomorphism
\begin{eqnarray}
\qquad \mathfrak{p}(2|2, \mathbb{C}) \! \! &\cong & \! \! \mathrm{span}_\mathbb{C}  \left\{L_{-1}, \,  L_0 \mp  \frac{1}{2} J_0,  \, L_1\mp J_1, \, G^\pm_{-1/2}, \, G^\mp_{-1/2}, \, G^\mp_{1/2}, \, G^\mp_{3/2}\right\}\\
&=& \! \! \mathfrak{g}_{\pm1} \smallsetminus \mathbb{C} J_0, \nonumber
\end{eqnarray}
is given as follows: $L_{-1}$, $L_0 - \frac{n}{2} J_0$, and $L_1 - n J_1$, for $n = \pm1$,  are given by the embedding of $\mathfrak{sl}(2, \mathbb{C})$ into $\mathfrak{sl}(2|2, \mathbb{C})$ via
\begin{equation}
\left( \begin{array}{cc}
a & b \\
c & -a
\end{array}\right)
\longrightarrow
\left( \begin{array}{cc|cc}
a & b & 0 &0 \\
c & -a & 0 & 0\\
\hline
0 & 0 & -a & -c \\
0 & 0 & -b & a
\end{array}\right),
\end{equation}
and the isomorphism defined by (\ref{sl2}).   Then for the odd components, we have 
\begin{equation*}
\begin{array}{cclccclc}
\left(\begin{array}{cc|cc}  
0 & 0&0 &0 \\
0 & 0 &0&0\\
\hline
0& 1 & 0 &0\\
-1 & 0 & 0 & 0 \end{array} \right) &\longleftrightarrow  &  G^\pm_{-1/2}, & \ \  & 
2 \left( \begin{array}{cc|cc}  
0 & 0 & 1 & 0 \\
0 & 0 &0 &0\\
\hline
0& 0 & 0 &0\\
0& 0& 0& 0 \end{array} \right)  &\longleftrightarrow&  G^\mp_{-1/2}, 
\\
\\
\left( \begin{array}{cc|cc}  
0 & 0 & 0 & -1 \\
0 & 0 &-1 &0\\
\hline
0& 0 & 0 &0\\
0& 0& 0& 0 \end{array} \right)  &\longleftrightarrow&  G^\mp_{1/2}, & & 
2 \left( \begin{array}{cc|cc}  
0 & 0 & 0 & 0 \\
0 & 0 &0 &1\\
\hline
0& 0 & 0 &0\\
0& 0& 0& 0 \end{array} \right)  &\longleftrightarrow&  G^\mp_{3/2},
\end{array}
\end{equation*} 
for $n = \pm 1$, respectively.  In addition we have 
\begin{equation*}
\left(\begin{array}{cc|cc}  
0 & 0&0 &0 \\
0 & 0 &0&0\\
\hline
0& 0 & \pm 1 &0\\
0 & 0 & 0 & \pm 1 \end{array} \right) \ \ \longleftrightarrow  \ \   J_0.
\end{equation*}
Thus the semi-direct product structure $\mathfrak{g}_{\pm 1} \cong  \mathfrak{gl}(1, \mathbb{C}) \times_\sigma \mathfrak{p}(2|2, \mathbb{C}) $ is given by the homomorphism
$\sigma : \mathfrak{gl}(1, \mathbb{C}) \longrightarrow \mathrm{Der}( \mathfrak{p}(2|2, \mathbb{C}))$ with 
$\sigma_{J_0} (g) = 0$ if $g$ is even and $\sigma_{J_0} (G^\pm_{k + \frac{1}{2}}) = \pm G^\pm_{k + \frac{1}{2}}$; and we have
\begin{equation}
\mathfrak{g}_{\pm 1} \cong \left\{ \left( \begin{array}{cc|cc} a & b & p &q\\
c & -a & q & r\\
\hline 
0 & s & -a+d & -c \\
-s & 0 & -b & a+d
\end{array}
\right) \; | \; a,b,c,d,p,q, r,s  \in \mathbb{C} \right\}.
\end{equation}

If $n \geq 2$ or $n \leq 2$, then $\{ 0 \} \oplus \mathfrak{g}_n^1$ is an abelian subalgebra of $\mathfrak{g}_n$.  Furthermore, $\mathfrak{g}_n$ is the semi-direct product of $\mathfrak{g}_n^0 \cong \mathfrak{sl}(2, \mathbb{C}) \oplus \mathfrak{gl}(1, \mathbb{C})$ acting on this abelian Lie superalgebra $\{0\} \oplus \mathfrak{g}_n^1$; i.e.,  $\mathfrak{g}_n \cong \mathfrak{g}_n^0 \times_{\sigma_n} ( \{0\} \oplus \mathfrak{g}_n^1)$, where for $n \geq 2$, we have $\sigma_n : \mathfrak{g}_n^0 \longrightarrow \mathrm{Der} (\{0\} \oplus \mathfrak{g}_n^1)$ is given by $(\sigma_n)_{L_{-1}} (G^-_{k - \frac{1}{2}}) = -k G^-_{k - \frac{3}{2}}$, $(\sigma_n)_{L_0 - \frac{n}{2} J_0} (G^-_{k - \frac{1}{2}}) = \left(-k + \frac{n+1}{2} \right)G^-_{k - \frac{1}{2}}$, $(\sigma_n)_{L_1 - n J_1} (G^-_{k - \frac{1}{2}}) = (-k  + n + 1)G^-_{k + \frac{1}{2}}$, and $(\sigma_n)_{J_0} (G^-_{k - \frac{1}{2}}) = -G^-_{k - \frac{1}{2}}$, and similarly for $n \leq 2$.  

Thus to recap, we have the following corollary:
\begin{cor}\label{infinitesimals-cor}
\begin{eqnarray}
\mathfrak{g}_0  = Lie(S^2\hat{\mathbb{C}}(0)) &\cong& \mathfrak{osp} (2|2, \mathbb{C}) \\
\mathfrak{g}_{\pm1} =  Lie(S^2\hat{\mathbb{C}}(\pm1 ))  & \cong &\mathfrak{gl}(1, \mathbb{C}) \times_\sigma   \mathfrak{p}(2|2, \mathbb{C}) \\
\mathfrak{g}_{n} =  Lie(S^2\hat{\mathbb{C}}(n)) & \cong & (\mathfrak{sl}(2, \mathbb{C}) \oplus \mathfrak{gl}(1, \mathbb{C})) \times_{\sigma_n} ( \{0\} \oplus \mathfrak{g}_n^1) ,
\end{eqnarray}
for $|n| \geq 2$, where $\{ 0 \} \oplus \mathfrak{g}_n^1$ is an abelian Lie superalgebra of even dimension zero and odd dimension $|n| + 2$.
\end{cor}

\section{The Lie supergroup structure of $\mathrm{Aut}(S^2\hat{\mathbb{C}}(n))$}

Let 
\begin{equation}
SL(2, \mbox{$\bigwedge_{*-2}^0$}) = \left\{ \left( \begin{array}{cc} a & b \\
c & d 
\end{array}
\right) \; | \; a,b,c,d \in \mbox{$\bigwedge_{*-2}^0$}, \; ad-bc = 1 \right\},
\end{equation}
and let $GL(1, \bigwedge_{*-2}^0) = (\bigwedge_{*-2}^0)^\times$.   The group $SL(2,\bigwedge_{*-2}^0) \times GL(1, \bigwedge_{*-2}^0)$ acts on $S^2\hat{\mathbb{C}} (n)$ as automorphisms as follows:
For each $n \in \mathbb{Z}$, 
\begin{equation}
\alpha = \left( \begin{array}{ccc} a & b & 0 \\
c & d & 0\\
0 & 0& \epsilon
\end{array}
\right) \in SL(2, \mbox{$\bigwedge_{*-2}^0$}) \times GL(1, \mbox{$\bigwedge_{*-2}^0$}),
\end{equation}
define 
\begin{equation}\label{action}
\alpha \cdot_n (z, \theta^+, \theta^-) = \left( \frac{az + b}{cz + d} , \ \theta^+ \epsilon (cz + d)^{n-1}, \ \theta^- \epsilon^{-1} (cz + d)^{-n-1} \right) 
\end{equation}
for $(z, \theta^+, \theta^-) \in (\bigwedge_{*>1}^0 \smallsetminus \{-d_B/c_B \} \times (\bigwedge_{*>1}^0)_S) \times (\bigwedge_{*>1}^1)^2$.

Acting in this way, $SL(2,\bigwedge_{*-2}^0) \times GL(1, \bigwedge_{*-2}^0)$ is a double cover of the even component of the Lie supergroup $\mathrm{Aut}(S^2\hat{\mathbb{C}}(n))$ for each $n \in \mathbb{Z}$.  In particular, the even component of the Lie supergroup $\mathrm{Aut} (S^2 \hat{\mathbb{C}} (n))$, denoted by $(\mathrm{Aut} (S^2 \hat{\mathbb{C}} (n)))^0$ is given as follows:  Let $id_{SL}$ denote the identity in $SL(2, \bigwedge_{*-2}^0)$ and let $id_{GL}$ denote the identity in $GL(1, \bigwedge^0_{*-2})$. Define
\begin{eqnarray}
K_{\bar{0}} &=& \langle(-id_{SL}, - id_{GL})\rangle \label{define-K0}\\
K_{\bar{1}} &=& \langle(-id_{SL},  id_{GL})\rangle \label{define-K1}
\end{eqnarray}  
which are both normal subgroups of order two in $SL(2, \bigwedge^0_{*-2}) \times GL(1, \bigwedge^0_{*-2})$.  Then from (\ref{action}), we have 
\begin{equation}
(\mathrm{Aut} (S^2 \hat{\mathbb{C}} (n)))^0\cong (SL(2, \mbox{$\bigwedge$}^0_{*-2}) \times GL(1, \mbox{$\bigwedge$}^0_{*-2})) / K_{\bar{n}} 
\end{equation}
where $\bar{n} \in \mathbb{Z}_2$ is the equivalence class of $n$ modulo $2$.

For $\mathfrak{g}$ a Lie superalgebra, let $\exp_{\bigwedge_{*-2}} (\mathfrak{g})$ denote its corresponding Lie supergroup over the Grassmann algebra $\bigwedge_{*-2}$.   Then 
\begin{eqnarray}
\exp_{\bigwedge_{*-2}}(\mathfrak{osp}(2|2, \mathbb{C})) &=& OSP(2|2, \mbox{$\bigwedge_{*-2}$}) \label{define-osp}\\ 
\exp_{\bigwedge_{*-2}}(\mathfrak{p}(2|2, \mathbb{C})) &=& P(2|2, \mbox{$\bigwedge_{*-2}$}). \label{define-p}
\end{eqnarray}   (These are the Lie supergroups denoted $C(2)$ and $P(1)$, respectively, studied in for instance \cite{D}.) We will also denote by $K_{\bar{0}}$ the image of $K_{\bar{0}}$ in  $OSP(2|2, \bigwedge_{*-2})$ given by the group generated by the negative of the $4\times 4$ identity matrix; and we will also denote by $K_{\bar{1}}$ the image of $K_{\bar{1}}$ in $P(2|2, \bigwedge_{*-2})$ given by the negative of the $4\times 4$ identity matrix.

For $n \in \mathbb{Z}_+$, the set $(\bigwedge_{*-2}^1)^n$ is an abelian group under addition.   If $n\geq 2$, then $(\bigwedge_{*-2}^1)^{n+2}$ acts as automorphisms of $S^2\hat{\mathbb{C}}(n)$ by 
\begin{multline}
(\psi^-_{n+1}, \dots, \psi^-_1, \psi^-_0) \cdot (z, \theta^+, \theta^-) = (z + \theta^+ (\psi^-_{n+1} z^{n+1}  + \cdots + \psi^-_1 z+ \psi^-_0), \\
\; \theta^+, \, \psi^-_{n+1} z^{n+1} + \cdots + \psi^-_1 z + \psi^-_0 + \theta^-),
\end{multline}
for $\psi^-_j \in \bigwedge_{*-2}^1$.  Analogously, if $n \leq 2$,  then $(\bigwedge_{*-2}^1)^{-n+2}$ acts as automorphisms of $S^2\hat{\mathbb{C}}(n)$ by 
\begin{multline}
(\psi^+_{-n+1}, \dots, \psi^+_1, \psi^+_0) \cdot (z, \theta^+, \theta^-) = (z + \theta^- (\psi^+_{-n+1} z^{-n+1} + \cdots + \psi^+_1 z\\
 + \psi^+_0), \, \psi^+_{-n+1} z^{-n+1}  + \cdots + \psi^+_1 z + \psi^+_0 +
 \theta^+, \,  \theta^-),
\end{multline}
for $\psi^+_j \in \bigwedge_{*-2}^1$.  

Thus we have the following theorem:
\begin{thm}\label{last-cor}  The automorphism groups of the N=2 superconformal super-Riemann spheres $S^2\hat{\mathbb{C}}(n)$, for $n \in \mathbb{Z}$, are
\begin{eqnarray}
\mathrm{Aut}(S^2\hat{\mathbb{C}}(0)) &\cong &OSP(2|2, \mbox{$\bigwedge$}_{*-2})/K_{\bar{0}} \\
\qquad \mathrm{Aut}(S^2\hat{\mathbb{C}}(\pm1)) &\cong & GL(1, \mbox{$\bigwedge$}_{*-2}^0) \ltimes (P(2|2, \mbox{$\bigwedge$}_{*-2})/K_{\bar{1}} )\\
\mathrm{Aut}(S^2\hat{\mathbb{C}}(n)) &\cong & \left((SL(2,  \mbox{$\bigwedge$}_{*-2}^0) \times GL(1, \mbox{$\bigwedge$}_{*-2}^0))/K_{\bar{n}} \right) \ltimes (\mbox{$\bigwedge$}_{*-2}^1)^{|n| + 2} ,
\end{eqnarray}
for $|n| \geq 2$, where $OSP(2|2, \mbox{$\bigwedge$}_{*-2})$ and $P(2|2, \mbox{$\bigwedge$}_{*-2})$ are as defined in (\ref{define-osp}) and (\ref{define-p}) above,  and the $K_{\bar{n}}$, for $n \in \mathbb{Z}$, are defined by (\ref{define-K0}) and (\ref{define-K1}).
\end{thm}

\end{document}